\documentclass[11pt,preprint]{article}

\usepackage{amsmath, latexsym, amsfonts, amssymb, amsthm}
\usepackage{graphicx, color,hyperref,dsfont,epsfig,caption,wrapfig,subfig}
\usepackage{pstricks,yfonts,mathalfa}
\usepackage{mathtools}
\usepackage{authblk}
\usepackage{movie15}
\hypersetup{
    linktoc=page,
    linkcolor=red,          
    citecolor=blue,        
    filecolor=blue,      
    urlcolor=cyan,
    colorlinks=true           
}

\numberwithin{equation}{section}

\newtheorem{theorem}{Theorem}[section]
\newtheorem{lemma}[theorem]{Lemma}
\newtheorem{prop}[theorem]{Proposition}

\newtheorem{rmk}[theorem]{Remark}
\newtheorem{defi}[theorem]{Definition}

\theoremstyle{definition}

\renewcommand{\tilde}{\widetilde}          
\DeclareMathSymbol{\leqslant}{\mathalpha}{AMSa}{"36} 
\DeclareMathSymbol{\geqslant}{\mathalpha}{AMSa}{"3E} 
\DeclareMathSymbol{\eset}{\mathalpha}{AMSb}{"3F}     
\renewcommand{\leq}{\;\leqslant\;}                   
\renewcommand{\geq}{\;\geqslant\;}                   

\def\bi{\begin{itemize}}
\def\ei{\end{itemize}}

\def\bnum{\begin{enumerate}}
\def\enum{\end{enumerate}}
\def\<#1{\langle #1 \rangle}

\newcommand{\norm}[1]{\left\lvert#1\right\rvert}
\newcommand{\expect}[1]{\mathbb{E}\left[#1\right]}
\usepackage{bm}
\usepackage{bbm}

\title{Liouville Conformal Field Theory \\on even-dimensional spheres}
\author{Baptiste Cercl\'e\thanks{baptiste.cercle@universite-paris-saclay.fr}}
\affil{Laboratoire de Math\'ematiques d'Orsay, Universit\'e Paris-Saclay.}
\date{}

\begin{document}

\maketitle
\abstract{Initiated by Polyakov in his 1981 seminal work, the study of two-dimensional Liouville Conformal Field Theory has drawn considerable attention over the past decades. Recent progress in the understanding of conformal geometry in dimension higher than two have naturally led to a generalization of Polyakov formalism to higher dimensions, based on conformally invariant operators: Graham-Jenne-Mason-Sparling operators and the $\mathcal{Q}$-curvature. 

This document is dedicated to providing a rigorous construction of Liouville Conformal Field Theory on even-dimensional spheres. This is done at the classical level in terms of a generalized \textit{Uniformization} problem, and at the quantum level thanks to a probabilistic construction based on log-correlated fields and Gaussian Multiplicative Chaos.
The properties of the objects thus defined are in agreement with the ones expected in the physics literature.}


\section{Introduction}
\subsection{Liouville Conformal Field Theory in higher dimension}

Providing a proper meaning to Liouville Conformal Field Theory (LCFT in the sequel) has been a fundamental issue since its introduction by Polyakov in his 1981 groundbreaking work~\cite{Pol81}. In this article, Polyakov describes a theory of summation over Riemannian metrics on a two-dimensional surface with fixed topology: formally speaking, this approach introduces a \textit{canonical way} of picking at random a geometry on a surface with fixed topology. To do so, a generalized path integral approach involving the the Liouville functional is being used, an approach which allows the introduction of a canonical random measure on such metrics, usually referred to as \textit{Liouville Quantum Gravity}. The problem of giving a rigorous meaning to Polyakov formalism has been an ongoing challenge for mathematicians over the past few decades and was successfully addressed thanks to the introduction of a probabilistic framework in a series of work initiated by David, Kupiainen, Rhodes and Vargas in \cite{DKRV16}. The culminating point of this programme aimed at providing a rigorous definition to Polyakov formalism is the proof of the DOZZ formula in \cite{KRV_DOZZ}, which is an explicit expression for the structure constants of LCFT, and that matches matches the one predicted in the physics literature~\cite{DO94,ZZ96} (and also derived by Teschner in~\cite{Tes95} and~\cite{Tes01, Tes04}).

On the other hand, the study of conformal geometry in dimension higher than two has considerably developed recently, with the introduction of higher-dimensional analogues of the Laplace operator and the Gauss curvature: the Graham-Jenne-Mason-Sparling (GJMS in what follows) operators and the $\mathcal{Q}$-curvature \cite{GJMS}. As we will see later, these operators play a role which is similar to the one of their two-dimensional analogues in the context of LCFT: it is therefore natural to expect that one can define LCFT in higher dimension by using the same framework as the one introduced in the two-dimensional setting. More generally, the topic of higher-dimensional Conformal Field Theory (CFT in the rest of the document) has regained attention lately, with for instance the AdS/CFT correspondence (introduced in the seminal work \cite{M99} by Maldacena) that establishes a duality between instances of string theories on Anti-de Sitter spaces and CFTs on their boundary, the most famous example establishing a correspondence between type IIB string theory and $\mathcal N=4$ supersymmetric Yang-Mills theory. 

The study of LCFT in a higher-dimensional context shares many interesting features with its two-dimensional analogue, being a non-perturbative theory which comes with an explicit formulation in terms of path integral. There is also good hope that, in full analogy with the two-dimensional case, the theory may be integrable in the sense that the three-point function of the theory may be explicitly computed (see the main result in \cite{FP18} as well as \cite[IV.D]{LO18}). The study of LCFT in a higher-dimensional context also enjoys a rich interplay with problems that naturally arise in the study of conformal geometry. Indeed, the classical aspect of LCFT corresponds to answering a problem of \textit{Uniformization}: does every compact even-dimensional manifold $(\mathcal M,g)$ carry a conformal metric with constant (negative) $\mathcal{Q}$-curvature? More generally, one could ask for a result similar to the one of Troyanov~\cite{T91}: is it possible to find a conformal metric with conical singularities and with prescribed $\mathcal{Q}$-curvature? The variational formulation of these two questions consists of finding critical points of a higher-dimensional analogue of the Liouville action functional~\eqref{Liouville_action}, which is the starting point of the higher-dimensional LCFT.

\subsection{The formalism of Liouville Conformal Field Theory}
As explained above, LCFT can be understood as a probabilistic framework that provides a natural way of picking at random a conformal structure on a $d$-dimensional compact manifold $(\mathcal{M},g)$. In other words, the theory should describe \textit{fluctuations} of the conformal geometry of a Riemannian manifold around its most natural structure, which corresponds to the solution of the Uniformization problem (\textit{i.e.} the metric with constant negative $\mathcal{Q}$-curvature $-\Lambda$). In order to study the difference between a given metric and this \lq\lq optimal metric" it is natural to consider the variational formulation of the problem of constant negative $\mathcal{Q}$-curvature: the solution of the Uniformization problem corresponds to the minimiser of the Liouville action functional, which takes in any even dimension the form~\cite{LO18}
\begin{equation}\label{Liouville_action}
    S_L(X,g)=\frac{d}{2(d-1)!\norm{\mathcal{M}}}\int_{\mathcal{M}}\left(X\mathcal{P}_gX+2\mathcal{Q}_gX+\frac{2}{d}\Lambda e^{dX}\right)\mathrm{d}\lambda_g
\end{equation}
for maps $X:\mathcal{M}\to\mathbb{R}$ and with $\mathrm{d}\lambda_g$ the volume form in the metric $g$. Therefore we can consider that a given conformal metric $e^{2X}g$ is close to \lq\lq the most canonical metric" when the value of the above functional is close to its minimal value. In the latter expression, we have introduced the geometric operators $\mathcal{P}_g$ and $\mathcal{Q}_g$ ---which are respectively the GJMS operator (a differential operator of order $d$) and the $\mathcal{Q}$-curvature (a scalar quantity)--- and that correspond to higher-dimensional generalizations of the Laplace-Beltrami operator and the Gauss curvature in the realm of conformal geometry.

In the quantum theory, we wish to provide a way to answering the following question: what does a canonical conformal structure on the manifold $(\mathcal{M},g)$ look like? This randomisation is done by introducing a random field $\phi_g$, the \textit{Liouville field}, whose law is described by formally setting, for $F$ bounded over the set of real-valued functions over $\mathcal{M}$,
\begin{equation}\label{Liouville_field}
    \expect{F(\phi_g)}=\frac{1}{\mathcal{Z}_g}\int F(X)e^{-S_L(X,g)}D_gX
\end{equation}
where $D_g$ stands for a \lq\lq Lebesgue measure" over maps $\mathcal{M}\to\mathbb{R}$ and $\mathcal{Z}_g$ is a renormalization factor referred to as the partition function. Picking randomly a \lq\lq canonical" Riemannian metric is therefore tantamount to considering the random metric $e^{2\phi_g}g$.
The form of Equation~\eqref{Liouville_field} shows us that the quantum field has a tendency to remain close to the classical field of the theory. As we will see below, before actually considering the above expression we will consider the quantization of the action, which corresponds to introducing quantum parameters in the Liouville functional.

One issue, at least at a mathematical level, is that the latter expression does not really make sense. Indeed, a first obstruction is to interpret the \lq\lq uniform measure" on fields $D_gX$; once this is done, one must provide a meaning to the Liouville action since the field $X$ is expected to be highly non-regular. The introduction of a probabilistic framework allows one to overcome these problems thanks to two objects that have become fundamental over the last decade: log-correlated fields and Gaussian Multiplicative Chaos (GMC in the sequel). The geometric flavour of the Liouville functional motivates the introduction of a log-correlated field to play the role of the map $X$ which appears in the Liouville action. This field generalizes the two-dimensional Gaussian Free Field, an analog of the Brownian motion for which the time variable now lives in a $d$-dimensional space, and which arises in many different contexts (statistical physics, theory of random surfaces, quantum field theory). As for the GMC, it can be understood as a random measure on Borel sets which can formally be written under the form $e^{\gamma X(z)}\mathrm{d}\lambda_g$ where $X$ is a log-correlated field and $\gamma$ is a positive real number. This writing is purely formal since the field $X$ is highly non-regular: an approximation procedure is necessary to give a proper meaning to this object. 

In this document we provide a rigorous definition of LCFT in a higher-dimensional context in the case where the manifold being considered is the even-dimensional sphere $\mathbb{S}^d$ (which naturally extends to a certain class of Riemannian manifolds). This is done at the classical level by deriving an existence and uniqueness statement for conformal metrics with constant negative $\mathcal{Q}$-curvature and prescribed conical singularities (Theorem~\ref{classical_conical}), and at the quantum level by providing a rigorous meaning to the path integral~\eqref{Liouville_field} in terms of the two probabilistic objects introduced above, log-correlated fields and GMC (see Theorem~\ref{partition_limit}). The construction thus obtained is consistent with the predictions of the physics literature~\cite{LO18} and with expected properties of a CFT~\cite{G96, R16} as stated in Proposition~\ref{conformal_covariance}.

\paragraph{Acknowledgements}I am very grateful to Antti Kupiainen for having supervised the internship during which this work has been undergone, as well as to the University of Helsinki for the support and hospitality provided while this problem was being investigated. I am also thankful to Vincent Vargas for having suggested this problem and for many fruitful discussions.

\section{Liouville Conformal Field Fheory in higher dimension: the classical theory}
The purpose of this section is to present the classical problem that lies behind LCFT in higher dimension. Namely, we first briefly review the notion of conformal symmetry in a space of dimension strictly higher than two. As we will see, the distinction between local and global transformations is no longer meaningful in contrast to the two-dimensional setting, where besides the $6$-parameter \textit{global conformal group} a special class of local conformal mappings exist: holomorphic mappings. We then present the conformally invariant operators which will be key in the formulation of LCFT: the GJMS operators and the $\mathcal{Q}$-curvature. With all these tools at hand a proper definition of the classical theory of LCFT in higher dimension, namely the study of a generalized Uniformization problem, is given.

\subsection{Conformal transformations in higher dimension}
In conformal field theory, some transformations play a central role: they are the so-called \textit{conformal maps}. Heuristically, a conformal map is a reparametrisation of the space that preserves angles. More rigorously, the notion of global conformal map can be defined on a Riemannian manifold $(\mathcal M,g)$ to be a diffeomorphism $\psi:\mathcal M\rightarrow \mathcal M$ such that the pulled-back metric $\psi^*g$ is conformally equivalent to $g$. In other words, there exists a smooth function $\varphi:\mathcal M\rightarrow \mathcal M$ such that $\psi^*g=e^{2\varphi}g$. A local conformal map is defined on a similar way by requiring this last proposition to hold only locally, that is on some open subset of $\mathcal M$.

Assume that $(\mathcal{M},g)$ is conformally equivalent to an Euclidean space. Then a theorem by Liouville asserts that when the dimension of $\mathcal{M}$ is strictly greater than $2$ a local conformal map must necessarily be an element of the $\frac{(d+1)(d+2)}2$-dimensional \textit{global conformal group}, made of \textit{M\"obius transforms}. This contrasts with the two-dimensional case where holomorphic mappings provide a large variety of local conformal maps.

Namely, let us denote by $\mathrm{d}\lambda$ the standard Lebesgue measure on $\mathbb{R}^d$ equipped with its flat metric $\norm{\mathrm d^2x}$. Then, on the Euclidean space $(\mathbb{R}^d\cup\lbrace\infty\rbrace,\norm{\mathrm d^2x})$, the M\"obius group is generated by four types of transformations: 
\begin{itemize}
    \item Translations $x\mapsto x+y$ for some $y\in\mathbb{R}^d$;
    \item Dilations $x\mapsto \rho x$ for some non-zero real $\rho$;
   \item Rotations $x\mapsto \Omega x$ where $\Omega$ is an element of the special orthogonal group $SO_n(\mathbb{R})$;
   \item Inversions $x\mapsto-\frac{\overline{x}}{\norm{x}^2}$ where $\overline{x}=(x_1,-x_2,...,-x_n)$.
\end{itemize}

For such maps $\psi:\mathbb{R}^d\cup\lbrace\infty\rbrace\rightarrow\mathbb{R}^d\cup\lbrace\infty\rbrace$, the flat-metric tensor  transforms as 
\[
\norm{\mathrm d^2x}\rightarrow \norm{\psi'(x)}\norm{\mathrm d^2x}
\]
where $\norm{\psi'(x)}:=\norm{\text{Jac}(\psi)(x)}^{\frac 2d}$ is called the conformal factor. Here Jac$(\psi)$ denotes the determinant of the Jacobian matrix of $\psi$. It may be easily seen that these M\"obius transforms satisfy a scaling property as follows:
\begin{lemma}
Let $\psi:\mathbb{R}^d\cup\lbrace\infty\rbrace\mapsto\mathbb{R}^d\cup\lbrace\infty\rbrace$ be a M\"obius transform, and $x,y$ be any two points in $\mathbb{R}^d$ (not mapped to $\infty$). Then
\begin{equation}
    \norm{\psi(x)-\psi(y)}=\norm{\psi'(x)}^{\frac 12}\norm{\psi'(y)}^{\frac12}\norm{x-y}.
\end{equation}
\end{lemma}
\begin{proof}
It suffices to check that the relation holds for the four types of basic transforms.
\end{proof}

\subsection{Conformal operators in higher dimension: GJMS operators and the $\mathcal{Q}$-curvature}
In order to define a (quantum) field theory that transforms covariantly under conformal transformations, it is natural to consider geometric objects enjoying this property. For this purpose in this subsection we present the operators that appear in the definition of the Liouville action~\eqref{Liouville_action}, the so-called \textit{GJMS operators} $\mathcal{P}_g$ and $\mathcal{Q}$\textit{-curvature} $\mathcal{Q}_g$. 

\subsubsection{Definition of the operators}
These objects may be understood as generalizations of the two-dimensional \textit{Laplace-Beltrami operator} $\Delta_g$ and \textit{Gauss curvature} $K_g$ in the realm of conformal geometry. Indeed, recall the transformation rules under conformal changes of the metric for the two-dimensional Laplace-Beltrami operator $\Delta_g$ and Gauss curvature $K_g$: if $g$ is any Riemannian metric on $\mathcal M$ and $g'=e^{2\varphi}g$ is a metric conformally equivalent to $g$, then the quantities $\Delta_{g'},K_{g'}$ and $\Delta_g,K_g$ are related by
\begin{equation}\label{conformal_2d}
\Delta_{g'}=e^{-2\varphi}\Delta_g \quad\text{and}\quad K_{g'}e^{2\varphi}=\Delta_g\varphi+K_g.
\end{equation}

In higher dimensions, a similar transformation rule under a conformal change of metrics also holds for the GJMS operators and the $\mathcal{Q}$-curvature; it takes the form
\begin{equation}\label{eq:cov_GJMS}
    \mathcal{P}_{e^{2\varphi}g}=e^{-d\varphi}\mathcal{P}_g\quad\text{and}\quad
    \mathcal{Q}_{e^{2\varphi}g}e^{d\varphi}=\mathcal{P}_g\varphi+\mathcal{Q}_g.
\end{equation}
This transformation rule allows us to form quantities that are invariants of a given conformal class of metrics on $\mathcal M$, such as:
\[
    \int_{\mathcal M}f\mathcal{P}_g f\mathrm{d}\lambda_g\quad\text{and}\quad \int_{\mathcal M}\mathcal{Q}_g\mathrm{d}\lambda_g
\]
where $\mathrm{d}\lambda_g$ is the volume form in the metric $g$ (when $g$ is the Euclidean metric on $\mathbb{R}^d$ we keep the notation $\mathrm{d}\lambda$ for the Lebesgue measure) and $f:\mathcal{M}\to\mathbb R$ is smooth and compactly supported. These quantities can be understood as generalizations of the Dirichlet energy and total curvature (which in dimension $2$ is nothing but $2\pi\chi(\mathcal M)$ where $\chi(\mathcal M)$ is the Euler characteristic of $\mathcal M$). Note that these quantities are topological invariant in two dimensions, but this is no longer the case in higher dimensions ---actually the topological structure of manifolds in higher dimensions can be really wild compared to the rigidity of the two-dimensional case (as illustrated for instance by the existence of so-called \textit{exotic spheres}~\cite{M56}).
The existence of operators that transform according to Equation~\eqref{eq:cov_GJMS} is due to Graham, Jenne, Mason and Sparling in~\cite{GJMS}, after whom the terminology \lq\lq GJMS operators\rq\rq was coined:
\begin{theorem}
Let $\mathcal M$ be a compact manifold of dimension $d\geq3$, and $N$ be such that: 
\begin{align*}
    1\leq N \leq\frac d2 &\quad\text{if $d$ is even}\\
    1\leq N &\quad\text{if $d$ is odd}.
\end{align*}
Then there exist conformally covariant differential operators $\mathcal{P}_{2N}$ of the form
\[
(-\Delta)^N+\text{lower order terms}
\]
that satisfy the transformation rule
\begin{equation}
    \mathcal{P}_{2N}(e^{2\varphi}g)=e^{-(\frac d2+N)\varphi}\mathcal{P}_{2N}(g)\circ e^{(\frac d2-N)\varphi}
\end{equation}
for any Riemannian metric $g$ on $\mathcal M$ and smooth $\varphi:\mathcal{M}\to\mathbb{R}$.
On the Euclidean space $(\mathbb{R}^{d},\vert\mathrm{d}^2x\vert)$ the operator $\mathcal{P}_{2N}\left(\vert\mathrm{d}^2x\vert\right)$ coincides with $(-\Delta)^N$, that is a power of the standard Laplace-Beltrami operator.
\end{theorem}
Of particular interest is the \textit{critical GJMS operator} $\mathcal{P}_d$ obtained by taking $N=\frac{d}{2}$ in the previous expression, which corresponds to the operator presented above. At this stage it is worth pointing out that this operator is actually well-defined only when the dimension of the manifold is even: as a consequence \textbf{we will always assume that we work in even dimensions in the sequel.} When the dimension of the manifold is odd, it is still possible to construct such operators but they are no longer differential operators but rather pseudo-differential operators (for instance it is given by $(-\Delta)^{3/2}$ for the $3$-dimensional Euclidean space). The initial approach used to define the $\mathcal{Q}$-curvature uses an argument of  \textit{analytic continuation} in the dimension from the more general scalar Riemannian invariant $\mathcal{Q}_{2N}$, defined as the terms of order zero of the GJMS operators (indeed, the latter is not well defined when we consider the critical GJMS operator since it annihilates constants). For a more detailed introduction see \cite{Bra95}. In the rest of the document we will work with the critical GJMS operator $\mathcal{P}_d(g)$ (resp. $\mathcal{Q}$-curvature $\mathcal{Q}_{d}(g)$) which we denote by $\mathcal{P}_g$ (resp. $\mathcal{Q}_g$).

The construction of these operators is not always explicit, and we know their expressions only in a few special cases, that is in low dimensions and for special manifolds: in dimension $d=4$ they correspond to the \textit{Paneitz operators} introduced in \cite{Paneitz}. In that case, their construction can be made explicit by setting
\[
\mathcal{P}_g=\Delta_g^2+\text{div}_g\left(\frac23 S_gg-2R_g\right)\circ \mathrm d\quad\text{and}\quad\mathcal{Q}_g=-\frac 1{12}\left(\Delta_gS_g-S_g^2+3\norm{R_g}^2\right),
\]
where $R_g$ is the Ricci tensor and $S_g$ the scalar curvature while $\mathrm d$ is the exterior derivative. Then the total curvature 
\[
\int_{\mathcal M} \mathcal Q_g\mathrm{d}\lambda_g
\]
is a conformal invariant of the manifold but unlike in two dimensions it is no longer a topological invariant.

In the sequel we will consider the special case where the manifold on which we work is the sphere of even dimension $\mathbb{S}^d$ equipped with its standard metric $g_0$\footnote{Quantities related to the sphere will be denoted with an index $0$.}. There are several reasons why we choose to work on this manifold:
\begin{itemize}
    \item The first one is technical: in this special case, the expression of the GJMS operator is explicit and given by 
    \[
\mathcal{P}_0\coloneqq\mathcal{P}_{g_0}=\prod_{k=0}^{\frac {n-2}2}\left(-\Delta_{g_0}+k(n-k-1)\right)
\]
where $\Delta_{g_0}$ is the standard Laplace-Beltrami operator on the sphere, while its $\mathcal{Q}$-curvature is constant and given by $\mathcal{Q}_{g_0}=(d-1)!$.
\item The sphere is conformally equivalent to the compactified space $\mathbb{R}^d\cup\lbrace\infty\rbrace$ equipped with the round metric
\begin{equation}\label{eq:round_metric}
\hat{g}=\frac{4}{\left(1+\norm{x}^2\right)^2}\vert\mathrm{d}^2x\vert
\end{equation}
which is obtained by stereographic projection (which is a conformal mapping) and that we will work with in what follows.
\item Eventually, working with the sphere actually covers the more general case of locally conformally flat manifolds, thanks to  a certain property of universality~\cite[Theorem 6]{K49}:
\begin{theorem}
Assume that $(\mathcal M,g)$ is a $d$-dimensional Riemannian manifold which is compact, simply connected and without boundary. If in addition $(\mathcal M,g)$ is locally conformally flat then it is conformally equivalent to the sphere $\mathbb{S}^d$.
\end{theorem}
\end{itemize}

As can be seen for instance in the case of the sphere, the GJMS operators are (in general) non-negative operators, in the sense that if $f$ is a smooth function on $\mathbb{S}^d$ then
\begin{equation}
(f,\mathcal{P}_gf)_g:=\frac{1}{\gamma_d}\int_{\mathbb S^d}f(y)\mathcal{P}_gf(y)\mathrm{d}\lambda_g(y)
\end{equation}
is non-negative and its value is independent of $g$ in the conformal class of $g_0$. Here the constant $\gamma_d$ is given by $\gamma_d:=\frac{(d-1)!\norm{\mathbb{S}^d}}2$. For general $(\mathcal{M},g)$ to ensure positivity of these GJMS operators one needs some assumptions to hold (\textit{e.g.} in dimension four having nonnegative Yamabe invariant and total $\mathcal{Q}$-curvature~\cite{G99}). Of special interest are the manifolds for which the kernel of this integral operator is made of constants, which is the case for the sphere and for a large class of manifolds.

Under these two assumptions (\textit{i.e.} that $\mathcal{P}_g$ is non-negative with $\ker \mathcal{P}_g=\lbrace\text{constants}\rbrace$) it is natural to work with the Sobolev space $H^{\frac{d}{2}}(\mathcal M,g)$ which can be defined as the Hilbert-space completion of the set of smooth functions with compact support in $(\mathcal M,g)$ with respect to the norm
\begin{equation}
    \norm{\norm{u}}_{H^{\frac{d}{2}}_g}:=\frac{1}{\gamma_d}\int_{\mathcal M} (u\mathcal{P}_gu+\gamma_du^2)\mathrm{d}\lambda_g.
\end{equation}

\subsubsection{Green's kernels for GJMS operators}
The next objects that we need to introduce are the Green's kernels $G_g$ associated to these GJMS operators, and which correspond to the kernels of the integral operators $\mathcal{P}_g^{-1}$. More precisely, these kernels are characterized as symmetric kernels $G_g:\mathcal{M}\times\mathcal{M}\to\mathbb{R}\cup\{\infty\}$ such that for any compactly supported smooth function $f$, one has
\[
    \begin{cases}
    \frac1{\gamma_d}\int_{\mathcal M} G_{g}(x,y)\mathcal{P}_gf(y)\mathrm{d}\lambda_g(y)&=f(x)-m_g(f)\\
    \int_{\mathcal M} G_g(x,y)\mathrm{d}\lambda_g(x) &= 0\\
\end{cases}
\]
where \begin{equation}
m_g(f)=\frac{1}{\lambda_g(\mathcal M)}\int_{\mathcal M}f(y)\mathrm{d}\lambda_g(y).
\end{equation}

In general it is possible to construct such a kernel; one of its features will be a logarithmic singularity on the diagonal, which means that $G_g(x,y)\sim\ln\frac{1}{d_g(x,y)}$ as $x$ gets close to $y$. However we won't discuss the general case in the present document and rather focus on the case where the manifold considered is the $d$-dimensional sphere $\mathbb{S}^d$, or equivalently by stereographic projection the projective space $\mathbb{R}^d\cup\lbrace\infty\rbrace$ (always with $d$ even). By doing so the expression of the Green's function becomes particularly simple:
\begin{prop}\label{Green_def}
For any Riemannian metric $g$ conformal to $\hat g$, consider the symmetric kernel $G_g(x,y):\mathbb{R}^d\times\mathbb{R}^d\rightarrow\mathbb{R}\cup\lbrace\infty\rbrace$ given by 
\begin{equation}
    G_g(x,y)=\ln\frac{1}{\norm{x-y}}-m_g\left(\ln\frac{1}{\norm{x-\cdot}}\right)-m_g\left(\ln\frac{1}{\norm{y-\cdot}}\right)+\theta_g
\end{equation}
where \[\theta_g=\frac{1}{\lambda_g(\mathbb{R}^d)^2}\int_{\mathbb{R}^d}\int_{\mathbb{R}^d}\ln\frac{1}{\norm{x-y}}\mathrm{d}\lambda_g(x)\mathrm{d}\lambda_g(y).
\]
Then $G_g$ has zero $\lambda_g$-mean:
    \[
    \int_{\mathbb{R}^d}G_g(x,y)\mathrm{d}\lambda_g(y)=0
    \]
    and is such that for any $f$ with compact support in $\mathbb{R}^d$ and $x\in\mathbb{R}^d$
    \[
     \frac{1}{\gamma_d}\int_{\mathbb{R}^d}G_g(x,y)\mathcal{P}_gf(y)\mathrm{d}\lambda_g(y)=f(x)-m_g(f).
    \] 
\end{prop}
\begin{proof}
The fact that $G_g$ has zero mean in the metric $g$ is straightforward. For the second point we start by considering the case where $g$ is the flat metric $\norm{d^2x}$ and set $G(x,y):=\ln\frac 1{x-y}$. Then for any compactly supported smooth function $f$ and positive $\varepsilon$
\begin{align*}
    &\int_{\mathbb{R}^d}\ln\frac{1}{\norm{x-y}}(-\Delta)^{\frac d2}f(y)\mathrm{d}\lambda(y)\\
    =&\int_{\mathbb{R}^d\setminus B(x,\varepsilon)}\ln\frac{1}{\norm{x-y}}(-\Delta)^{\frac d2}f(y)\mathrm{d}\lambda(y)+\int_{B(x,\varepsilon)}\ln\frac{1}{\norm{x-y}}(-\Delta)^{\frac d2}f(y)\mathrm{d}\lambda(y).
\end{align*}
The function $y\mapsto\ln\frac{1}{\norm{x-y}}$ is smooth outside of $B(x,\varepsilon)$, therefore we can use integration by parts to get that
\begin{align*}
    &\int_{\mathbb{R}^d\setminus B(x,\varepsilon)}\ln\frac{1}{\norm{x-y}}(-\Delta)^{\frac d2}f(y)\mathrm{d}\lambda(y)=\int_{\mathbb{R}^d\setminus B(x,\varepsilon)}(-\Delta)_y^{\frac d2}\ln\frac{1}{\norm{x-y}}f(y)\mathrm{d}\lambda(y)+\\
    &\sum_{k=0}^{\frac{d}{2}-1}\int_{\partial B(x,\varepsilon)} \left((-\Delta)^kf(y)\frac{\partial}{\partial_n}((-\Delta)_y^{\frac d2-k-1}\ln\frac{1}{\norm{x-y}})-(-\Delta)_y^k\ln\frac{1}{\norm{x-y}}\frac{\partial}{\partial_n}(-\Delta)^{\frac d2-k-1}f(y))\right)\mathrm{d}\lambda_{\partial}(y).
\end{align*}
Now, a standard Laplacian computation shows the expressions
\begin{align}
    (-\Delta)_y^m\ln\frac{1}{\norm{x-y}}=\frac{2^{m-1}(m-1)!}{\norm{x-y}^{2m}} \prod_{k=1}^m (d-2k) & \quad\text{for $m$ positive integer,}\\
    \frac{\partial}{\partial_n}(-\Delta)_y^m\ln\frac{1}{\norm{x-y}}=\frac{2^mm!}{\norm{x-y}^{2m+1}}\prod_{k=1}^m (d-2k) & \quad\text{for $m$ integer}.
\end{align}
As a consequence the first term in the expression above vanishes. Similarly, of the terms that appear in the sum, the only whose order is that of the volume of $\partial B(x,\varepsilon)$ as $\varepsilon\rightarrow0$ is given by $k=0$.
Therefore
\[
\lim\limits_{\varepsilon\rightarrow0} \int_{\mathbb{R}^d\setminus B(x,\varepsilon)}\ln\frac{1}{\norm{x-y}}(-\Delta)^{\frac d2}f(y)\mathrm{d}\lambda(y)=f(x)\frac{\left(2^{\frac d2}(\frac d2-1)!\right)^2}{2} \norm{\mathbb{S}^{d-1}}.
\]
The same reasoning applies to the term $\int_{B(x,\varepsilon)}\ln\frac{1}{\norm{x-y}}\Delta^{\frac d2}f(y)\mathrm{d}\lambda(y)$, which is negligible as $\varepsilon\rightarrow0$.
To finish up, notice that for $d$ an even integer, $\frac{\left(2^{\frac d2}(\frac d2-1)!\right)^2}{2} \norm{\mathbb{S}^{d-1}}=\gamma_d$; as a consequence we have proved that for smooth and compactly supported $f$:
\begin{equation}
\frac1\gamma_d \int_{\mathbb{R}^d}\ln\frac{1}{\norm{x-y}}(-\Delta)^{\frac d2}f(y)\mathrm{d}\lambda(y)=f(x).
\end{equation}

Therefore if $g$ is in the conformal class of $\hat g$: 
\begin{align*}
   &\frac1\gamma_d\int_{\mathbb{R}^d}\left(\ln\frac{1}{\norm{x-y}}-m_g\left(\ln\frac{1}{\norm{x-\cdot}}\right)-m_g\left(\ln\frac{1}{\norm{y-\cdot}}\right)+\theta_g\right)\mathcal{P}_gf(y)\mathrm{d}\lambda_g(y) \\
   &= \frac1\gamma_d\int_{\mathbb{R}^d}\left(\ln\frac{1}{\norm{x-y}}-m_g\left(\ln\frac{1}{\norm{x-\cdot}}\right)-m_g\left(\ln\frac{1}{\norm{y-\cdot}}\right)+\theta_g\right)(-\Delta)^{\frac d2}f(y)\mathrm{d}\lambda(y)\\
   &=f(x) + \left(\theta_g-m_g\left(\ln\frac{1}{\norm{x-\cdot}}\right)\right)\int_{\mathbb{R}^d}(-\Delta)^{\frac d2}f(y)\mathrm{d}\lambda(y)-\frac{1}{\lambda_g(\mathbb{R}^d)}\iint_{\mathbb{R}^d\times\mathbb{R}^d}\ln\frac{1}{\norm{y-z}}(-\Delta)^{\frac d2}f(y)\mathrm{d}\lambda(y)\mathrm{d}\lambda_g(z)\\
   &=f(x) -\frac{1}{\lambda_g(\mathbb{R}^d)}\int_{\mathbb{R}^d}f(z)\mathrm{d}\lambda_g(z)\\
   &=f(x)-m_g(f).
\end{align*}
\end{proof}

Before moving on, we raise some elementary statements in the case where $g$ is the round metric $\hat g$.  To start with we provide an explicit expression for the Green's kernels $G_{\hat g}$:
\begin{lemma}\label{lemma:round}
Let $\hat{g}$ be the round metric on $\mathbb{R}^d$. Then
\begin{equation}
G_{\hat{g}}(x,y)=\ln\frac{1}{\norm{x-y}}-\frac{1}{4}\left(\ln\hat{g}(x)+\ln\hat{g}(y)\right)+C_{\hat{g}}
\end{equation}
where $C_{\hat{g}}$ is some real constant.
\end{lemma}
\begin{proof}
According to the conformal change of metric formula, and since the round metric $\hat g$ has constant $\mathcal{Q}$-curvature equal to $(d-1)!$ while the flat one $\norm{d^2x}$ has zero $\mathcal{Q}$-curvature, we know that 
\[ (d-1)!\hat{g}(x)^{\frac d2}=(-\Delta)^{\frac d2}\frac12\ln \hat{g}(x).
\]
As a consequence the function given by 
\[
    F(x):=2m_{\hat{g}}(\frac{1}{\ln\norm{x-\cdot}})-\frac12\ln\hat{g}(x)+\ln 2
\]
is such that $\Delta^{\frac d2}F=0$. Moreover $F(x)$ converges toward $0$ as $\norm{x}\rightarrow+\infty$ and is continuous at $x=0$. Since $F$ depends only on $\norm{x}$ (because this is the case for $\hat g$) we can use Lemma~\ref{radial_harmonic} to conclude that we must have $F=0$, which yields the result.
\end{proof}
Our next statement highlights the relationship between Green's kernels and M\"obius transforms. Indeed, Green's kernels share a property of covariance under M\"obius transforms, which takes for the round metric $\hat g$ the following form:
\begin{lemma}\label{Mobius_Green}
Let $\psi$ be a M\"obius transform of $\mathbb{R}^d$ and $x,y$ be any two points in $\mathbb{R}^d$ not mapped to $\infty$. Then
\begin{equation}
    G_{\hat{g}}(\psi(x),\psi(y))=G_{\hat{g}}(x,y)-\frac{1}{4}\left(\ln \norm{\psi'(x)}^2\frac{\hat{g}(\psi(x))}{\hat{g}(x)}+\ln \norm{\psi'(y)}^2\frac{\hat{g}(\psi(y))}{\hat{g}(y)}\right).
\end{equation}
\end{lemma}
\begin{proof}
It suffices to prove the result for the four types of basic M\"obius transforms.

The proof is obvious if $\psi$ is a translation. \\
Let us now assume that $\psi$ is either a dilation, a rotation or an inversion. In that case the function
\[
m_{\hat{g}}(\frac{1}{\ln\norm{\psi(x)-\cdot}})=\frac{1}{\norm{\mathbb{S}^d}}\int_{\mathbb{R}^d}\left( \ln\frac{1}{\norm{x-y}}-\frac{1}{2}\ln\norm{\psi'(x)}-\frac{1}{2}\ln\norm{\psi'(y)}\right)\norm{\psi'(y)}^d\hat{g}^{\frac{d}{2}}(\psi(y))\mathrm{d}\lambda(y)
\] is radial, and we have the property that 
\[
(-\Delta)^{\frac d2}\left(m_{\hat{g}}(\frac{1}{\ln\norm{\psi(x)-\cdot}})+\frac{1}{2}\ln\norm{\psi'(x)}\right)=\frac12(d-1)!\norm{\psi'(x)}^d\hat{g}^{\frac{d}{2}}(\psi(x)).
\]
Now since the metric given by $\hat{g}_{\psi}=\norm{\psi'}\hat{g}\circ\psi$ is nothing but the pull-back measure of $\hat{g}$ by the M\"obius transform $\psi$, the $\mathcal{Q}$-curvature of this metric is the same as the one of $\hat{g}$, that is $(d-1)!$. As a consequence $(d-1)!(\norm{\psi'}^2\hat{g}\circ\psi)(x)^{\frac d2}=(-\Delta)^{\frac d2}\frac12\ln \norm{\psi'}^2\hat{g}\circ\psi(x)$, whence
\[
m_{\hat{g}}(\frac{1}{\ln\norm{\psi(x)-\cdot}})=-\frac{1}{2}\ln\norm{\psi'(x)}+\frac{1}{4}\ln \norm{\psi'(x)}^2\hat{g}(\psi(x))
+F(x)=\frac{1}{4}\ln\hat{g}(\psi(x))+F(x)\]
where $F$ is radial, smooth and satisfies $\Delta^{\frac d2}F=0$. Since it vanishes for $\norm{x}\rightarrow+\infty$ we can conclude by Lemma~\ref{radial_harmonic} that $F=0$.
\end{proof}
We also state a counterpart result for Lemma~\ref{Mobius_Green}:
\begin{lemma}\label{lemma:mobius_GFF}
In the setting of Lemma~\ref{Mobius_Green}, set $\hat g_\psi(x)\coloneqq \norm{\psi'(x)}^2\hat g(\psi(x))$. Then
\begin{equation}
G_{\hat g}(\psi(x),\psi(y))=G_{\hat g}(x,y)-m_{\hat g_\psi}(G_{\hat g}(x,\cdot))-m_{\hat g_\psi}(G_{\hat g}(\cdot,y))+C_{\hat g,\psi}
\end{equation}
for some constant $C_{\hat g,\psi}$.
\end{lemma}
\begin{proof}
By conformal invariance of the GJMS operator it is readily seen that \begin{equation}
G_{\hat g}(\psi(x),\psi(y))=G_{\hat g_\psi}(x,y).
\end{equation}
Now thanks to Proposition~\ref{Green_def} we can write that 
\[
G_{\hat g_\psi}(x,y)=G_{\hat g}(x,y)-m_{\hat g_\psi}(G_{\hat g}(x,\cdot))-m_{\hat g_\psi}(G_{\hat g}(\cdot,y))+C_{\hat g,\psi}.
\]
\end{proof}
\subsection{Classical LCFT on the sphere and Uniformization}\label{classical_Liouville}
In this last subsection, we are interested in studying the classical LCFT on the $d$-dimensional sphere. In other words we will consider the Uniformization problem on the sphere by investigating the existence of a conformal metric which has constant negative $\mathcal{Q}$-curvature.

\subsubsection{The Classical Liouville action functional}
As explained above, the starting point of the classical LCFT is to find a metric conformally equivalent to $g$ and which has constant negative curvature $-\Lambda$. Since a conformal metric can be put under the form $e^{2\phi}g$, by Equation~\eqref{eq:cov_GJMS} the latter is tantamount to saying that the conformal factor $\phi$ is a solution of a higher-dimensional Laplace equation:
\begin{equation}\label{eq:Laplace}
    \mathcal{P}_g\phi+\mathcal{Q}_g=-\Lambda e^{d\phi}.
\end{equation}
The variational formulation of this problem just corresponds to saying that $\frac 1b\phi$ is a minimiser of the Liouville action functional~\eqref{Liouville_action}
\[
 S_L(X,g)=\frac{d}{4\gamma_d}\int_{\mathbb{S}^d}\left(X\mathcal{P}_gX+2Q\mathcal{Q}_gX+\Lambda b^{-2} e^{dbX}\right)\mathrm{d}\lambda_g,
 \]
  where in the classical theory the parameter $Q$ is chosen to be equal to $\frac1b$. By doing so the action is classically Weyl invariant, in the sense that 
 \[
 S_L(X-Q\sigma,e^{2\sigma}g)-S_L(X,g)
 \]
does not depend on $X$. Put differently, the metric $e^{2bX}g$ that minimizes the above functional is actually independent of the background metric in a given conformal class. Note that here we have introduced a positive parameter $b$ instead of the standard action; the reason why we do so will be made clear later when we will consider the quantum theory associated to this classical one.
 
The existence of such metrics is not known in general; however if we do not consider anymore the assumption on the sign of the curvature an answer to this Uniformization problem exists for a large class of Riemannian manifolds  (\cite[Theorem 1.1]{ND07}):
\begin{theorem}\label{classical_curvature}
Let $(\mathcal M,g)$ be a compact $d$-dimensional (with $d$ even) Riemannian manifold and assume that:\begin{itemize}
    \item The kernel of the GJMS operator $\mathcal{P}_g$ is made of constant functions.
    \item The conformal invariant $\kappa_{\mathcal M}:=\int_{\mathcal M}\mathcal{Q}_g\mathrm{d}\lambda_g$ is not an integer multiple of $\kappa_{\mathbb{S}^d}=(d-1)!\norm{\mathbb{S}^d}$.
\end{itemize} 
Then $(\mathcal M,g)$ admits a conformal metric with constant $\mathcal{Q}$-curvature.
\end{theorem}

The particular case of the higher-dimensional sphere is also well understood since we then know explicitly the answers to this problem (\cite[Theorem 1.1]{YCY97}):
\begin{theorem}\label{classical_sphere}
Assume that $g$ is a metric on the even-dimensional sphere $\mathbb{S}^d$, conformally equivalent to $g_0$. If the $\mathcal{Q}$-curvature of $g$ is constant then necessarily it is of the form
$g=\frac{\text{Vol}_g(\mathbb{S}^d)}{\norm{\mathbb{S}^d}}\psi^*g_0$ for some M\"obius transform $\psi$ of the sphere.
\end{theorem}

\subsubsection{The Liouville functional with conical singularities}

Coming back to the case of the sphere, we see that there is an obvious obstruction for the existence of a metric with constant \textit{negative} $\mathcal{Q}$-curvature: since the total curvature is a conformal invariant, taking the integral over the sphere of the Laplace equation~\eqref{eq:Laplace} yields the equality $2\gamma_d=-\Lambda \text{Vol}_{e^{2\phi}g}(\mathbb{S}^d)$, which is impossible if $\Lambda$ is positive. A natural way to overcome this obstruction is to extend the range of definition of the conformal factor by allowing it to have logarithmic singularities, that is by allowing the fields to behave like
$\phi\sim 2\chi_k\ln\frac{1}{\norm{x-x_k}}$ near to the marked points $(x_k,\chi_k)\in\mathbb{S}^d\times\mathbb{R}$. The geometric interpretation of such singularities is a conical singularity on the underlying manifold at the point $x_k$ whose angle is prescribed by the weight $\chi_k$. 

This operation amounts to adding to the action \textit{conical singularities}, which is achieved by considering (a regularization of)
\[
S^{(\bm x,\bm{\chi})}_L(X,g):=S_L(X,g)-d\sum_{k=1}^N\chi_k X(x_k)
\]
where $\bm x:=(x_1,...,x_N)$ are points in $\mathbb{S}^d$ and $\bm{\chi}:=(\chi_1,...,\chi_N)$ are (real) weights. This operation is the classical analog of inserting Vertex Operators in the partition function. By doing so, the field equations are transformed in the following way:
\begin{equation}\label{eq:Liouville_equ_conical}
\mathcal{P}_g \phi +\mathcal{Q}_g + \Lambda e^{d\phi}=2\gamma_d \sum_{k=1}^N \chi_k\delta_g(x-x_k).
\end{equation}
Solutions of this equation correspond conformal metrics $e^{2\phi}g$ with constant negative $\mathcal{Q}$-curvature $-\Lambda$ and conical singularities prescribed by the $(\bm x,\bm{\chi})$. Their existence is subject to at least two conditions:
\begin{itemize}
    \item Integrability of the volume form near the singularities, which implies that $\chi_k< \frac{1}{2}$.
    \item Integrating Equation~\eqref{eq:Liouville_equ_conical} implies that $\Lambda\int e^{d\phi}\mathrm{d}\lambda_g=\sum_k\chi_k-\frac{1}{2\gamma_d}\int Q_{g}\mathrm{d}\lambda_g$. In particular for the sphere $\mathbb{S}^d$ we get the relation $\Lambda\int e^{d\phi}\mathrm{d}\lambda_g=\sum_k\chi_k-1$. If we choose $\Lambda$ to be positive this implies that $\sum_k\chi_k>1$.
\end{itemize}
In particular for the sphere to have a conformal metric with constant negative $\mathcal{Q}$-curvature one needs to prescribe at least three conical singularities. The bounds that appear are the classical analogues of the \textit{Seiberg bounds}~\cite{Sie90} that occur in quantum field theory.

In the higher-dimensional theory, we are able to provide a result of existence and uniqueness for this problem (much weaker than the two-dimensional one by Troyanov in \cite{T91} but with an elementary proof) in the very special case
of the $d$-dimensional sphere. This can be stated in the following way:
\begin{theorem}\label{classical_conical}
Let $N$ be any integer $N$ and $\Lambda$ be a positive real number. Let $x_1,\cdots,x_N$ be distinct on $\mathbb{S}^d$, and assume that $\chi_1,\cdots,\chi_N$ are real numbers such that 
\begin{equation}
    \forall k, \chi_k<\frac{1}{2}\quad\text{and}\quad \sum_{k}\chi_k>1.
\end{equation}
Then there exists a unique metric $g=e^{2\phi_0}g_0$ on the sphere $(\mathbb{S}^d,g_0)$ such that:
\begin{itemize}
    \item $g$ has conical singularities of weight $\chi_k$ at the point $x_k$ for any $k$.
    \item $g$ has constant negative curvature $-2\gamma_d\Lambda$.
    \item $\phi_0-2\sum_{k=1}^N\chi_k\ln\frac{1}{\norm{x-x_k}}$ is in the Sobolev space $H^{\frac{d}{2}}(\mathbb{S}^d,g_0)$.
\end{itemize}
Put differently, there exists a unique $h_0\in H^{\frac d2}(\mathbb{S}^d,g_0)$ such that $\phi_0\coloneqq h_0+2\sum_k\chi_{k=1}^N\ln\frac{1}{\norm{x-x_k}}$ is a variational solution of the problem 
\[
\mathcal{P}_0 \phi_0 +(d-1)! + 2\gamma_d\Lambda e^{d\phi_0}=2\gamma_d \sum_{k=1}^N \chi_k\delta(x-x_k).
\] 
Moreover $\phi_0$ is smooth outside of its singular points. 
\end{theorem}

This result has a straightforward generalization to 
any $d$-dimensional compact Riemannian manifold without boundary which is simply connected and locally conformally flat.

The proof of this result relies on a variational approach (with a functional which is actually not the Liouville one) of the problem, which involves a Moser-Trudinger-type inequality which takes the following form:
\begin{prop}\label{Mos_Tru}
In the setting of Theorem~\ref{classical_conical}, assume only that the $\bm{\chi}$ are such that $\chi_k<\frac12$ for all $k$. Then there exist positive constants $c$ and $C$, depending on $\bm{\chi}$, such that for any $f\in H^{\frac d2}(\mathbb{S}^d,g_0)$:
\begin{equation}
    \ln \int_{\mathbb{S}^d}e^{f+2d\sum_k\chi_k\ln\frac{1}{\norm{x-x_k}}}\mathrm{d}\lambda_{g_0} \leq c+C\int_{\mathbb{S}^d}f\mathcal{P}_0f\mathrm{d}\lambda +\int_{\mathbb{S}^d}f\mathrm{d}\lambda_{g_0}.
\end{equation}
When $\chi_k=0$ for all $k$ one can take $c=0$ and $C=\frac{1}{2d!}$ (this corresponds to the main result in \cite{B93}).
\end{prop}
This inequality is easily derived from the result in \cite[Theorem 1]{B93}; we prove these two statements in Appendix~\ref{classical_proofs}. We will also see below that when considering the quantization of the Liouville functional (which is the purpose of the next section), it is natural to expect that the semi-classical limit of the model (which corresponds to letting the \lq\lq quantum parameters" go to zero) coincides with the solution of the above problem. This is indeed the case in dimension two~\cite{LRV19} and the extension of this result will be discussed in Subsection~\ref{subsec:semi_classical}.

\section{Quantization of the action: Liouville Conformal Field Theory on the higher-dimensional sphere}

The purpose of this section is to provide a rigorous meaning to the quantization of the classical theory exposed in the previous section. This is done by giving a definition of the random field formally introduced by using a path integral approach (recall Equation~\eqref{Liouville_field}) thanks to a probabilistic framework. In this section we will work with the Liouville action whose expression is given by
\begin{equation}
     S_L(X,g)=\frac{d}{4\gamma_d}\int_{\mathbb{S}^d}\left(X\mathcal{P}_gX+2Q\mathcal{Q}_gX+\frac{4\gamma_d}{d}\mu e^{dbX}\right)\mathrm{d}\lambda_g
\end{equation}
which corresponds to the quantization of the classical action. In the above expression the coupling constant $b\in(0,1)$ corresponds to the \lq\lq level of randomness" considered (the deterministic theory corresponds to the limit $b\rightarrow0$ under suitable renormalization, usually referred to as the semi-classical limit). Here $Q=\frac1b+b$ is the background charge; note that it differs from its classical value by the parameter $b$ which account for the \lq\lq quantum corrections" that have to be added for the model to be well defined (see below). Eventually $\mu>0$ is the cosmological constant.

\subsection{Probabilistic background}
To start with, we present the probabilistic background that we will need in order to give a meaning to the Liouville action functional. 
\subsubsection{Log-correlated fields}
The first term we need to interpret is the measure element $D_gX$ that appears in the path integral. But instead of considering this measure element we will rather consider the Gaussian measure that is formally defined by 
\[
\exp\left(-\frac{d}{4\gamma_d}\int_{\mathbb{S}^d} X\mathcal{P}_gX\right)D_gX.
\]
The form of the measure element is indeed reminiscent of a Gaussian measure; it may be interpreted as a Gaussian measure on the collection of the $\left((X,\mathcal{P}_gf)\right)$ indexed by $f\in H^{\frac d2}(\mathbb{S}^d,g)$ under which 
\begin{equation}
    \frac{d}{2\gamma_d}\expect{(X,\mathcal{P}_gf)(X,\mathcal{P}_gh)}=(f,\mathcal{P}_gh)
\end{equation}
for $f$ and $h$ in $H^{\frac d2}(\mathbb{S}^d,g)$. Now if we formally think of $X$ as a function and exchange expectations and integrals, the latter can be rewritten as 
\[
    \frac{d}{2\gamma_d}\int_{\mathbb{S}^d}\int_{\mathbb{S}^d}\mathcal{P}_gf(x)\expect{X(x)X(y)}\mathcal{P}_gh(y)\mathrm{d}\lambda_g(x)\mathrm{d}\lambda_g(y)=\int_{\mathbb{S}^d}f(x)\mathcal{P}_gh(x)\mathrm{d}\lambda_g(x).
\]
In particular if we set $\expect{X(x)X(y)}=\frac2d G_{g}(x,y)$ we get the desired result. This leads us to the introduction of a log-correlated field as follows (where we view the sphere as $\mathbb{R}^d\cup\{\infty\}$).
\begin{defi}
Consider $g$ a metric on $\mathbb{R}^d\cup\{\infty\}$ conformally equivalent to $\hat g$. We define a log-correlated field $X_g$ as a centered Gaussian random distribution with covariance kernel given by $\frac2d G_{g}$ where $G_{g}$ is the Green's kernel of the GJMS operator defined by Proposition~\ref{Green_def}.
Put differently, for $x,y\in\mathbb{R}^d$:
\begin{equation}
    \expect{X_g(x)X_g(y)}=\frac2d G_{\psi^*g}(\psi (x),\psi (y)).
\end{equation}
\end{defi}

The existence of such a field is ensured by the non-negativity of its kernel (more details on its construction can be found in the review \cite{DRSV14}); it can be shown - for instance by adapting the reasoning conducted in \cite[Section 4.3]{Dub07} - that it is possible to work in a probability space on which the random field $X_g$ lives almost surely in the dual space $H^{-\frac d2}(\mathbb{S}^d,g)$ of $ H^{\frac d2}(\mathbb{S}^d,g)$, which we will always assume in the sequel.

From the definition of the covariance kernel as provided before, one observation is that the mean-value of the field is zero almost surely; also under a conformal change of metric $g'$ the two fields 
\[
X_{g'} \quad \text{and} \quad X_g-m_{g'}(X_g)
\]
have same law. This means that if we interpret the measure $\propto e^{-\frac{d}{4\gamma_d}(X,\mathcal{P}_gX)}D_gX$ as the probability measure of a zero-mean log-correlated fields we lose one degree of freedom for the field (which corresponds to the kernel of the GJMS operator). To address this issue, we will add a constant term to the field which will be chosen uniformly according to the Lebesgue measure on $\mathbb{R}$ (which should be understood as a Gaussian measure with infinite variance).

To summarise and by viewing the sphere as the compactified space $\mathbb{R}^d\cup\{\infty\}$, we may interpret the measure 
\[
\exp\left(-\frac{d}{4\gamma_d}\int_{\mathbb{S}^d} X\mathcal{P}_gX\right)D_gX
\]
as the image by $(X_g,c)\mapsto X_g+c$ of the tensor product $d\mathbb{P}_g(X)\otimes dc$, where $d\mathbb{P}_g(X)$ denotes the measure associated to the log-correlated field $X_g$ and $dc$ refers to the Lebesgue measure on $\mathbb{R}$. Note that the measure thus defined is an infinite measure.

\subsubsection{Gaussian Multiplicative Chaos}
The next step in interpreting the Liouville action~\eqref{Liouville_action} is to make sense of the term 
\[
\int_{\mathbb{S}^d} e^{dbX(x)} \mathrm{d}\lambda_g(x)
\] that appears in the expression of the Liouville functional.
Indeed, the lack of regularity of the log-correlated field $X$ prevents us from providing a rigorous meaning to the term $e^{dbX(x)}$ viewed as a well-defined function. Nonetheless, the theory of Gaussian Multiplicative Chaos (GMC), first introduced in the seminal work~\cite{K85} by Kahane, explains that one can make sense of $e^{dbX(x)} \mathrm{d}\lambda_g(x)$ as a \textit{random measure}. This is done thanks to a regularization procedure, by considering a smooth approximation of the field $X$ and taking an appropriate scaling limit. Additional details on this construction can be found \textit{e.g.} in~\cite{Ber17} or~\cite{RV16} and the references therein.

For convenience, we will consider here the realization of the sphere as the space $\mathbb{R}^d\cup\lbrace\infty\rbrace$ equipped with the round metric $\hat g$ (this corresponds to considering $X_g$ with covariance kernel $\frac 2dG_g$ as in Proposition~\ref{Green_def}). We consider the regularization of the field given by its average on $d-1$-dimensional spheres: for positive $\varepsilon$ we set
\begin{equation}\label{eq:average}
    X_{g,\varepsilon}:=\frac{1}{\norm{\mathbb{S}^{d-1}}}\int_{\mathbb{S}^{d-1}} X_{g}(x+\varepsilon y)\mathrm{d}\lambda_{\partial}(y)
\end{equation}
where $\mathrm{d}\lambda_\partial$ is the Lebesgue measure on $\mathbb{S}^{d-1}\subset\mathbb R^d$.

As can be checked from the definition of $X_g$, the variance of the Gaussian centered random variable $X_{g,\varepsilon}(x)^2$ is of order $\frac{2}{d}\ln\frac{1}{\varepsilon}$, so we have that $\expect{e^{db X_{g,\varepsilon}(x)^2}}$ is of order $\varepsilon^{-db^2}$. This should at least motivate the following statement (which is standard in the theory of GMC, see for instance~\cite[Theorem 2.3]{RV16}):
\begin{prop}\label{prop:GMC}
\hspace{0.5cm}\\
For positive $\varepsilon$ and $b\in (0, 1)$, define the random measure $M_{b,g}^{\varepsilon}(\mathrm d x):=\varepsilon^{db^2}e^{db(X_{g,\varepsilon}(x)+\frac{Q}{2}\ln g)}\mathrm{d}\lambda(x)$. Then the following limit exists in probability (the limit is taken in the sense of weak convergence of measures):
\[
M_{b,g}:=\lim\limits_{\varepsilon\rightarrow0}M_{b,g}^{\varepsilon}.
\]
\end{prop}
As explained in \cite{RV16} the above limit is actually independent of the choice of regularization of the field; the choice we made is mostly aimed at keeping the document as clear as possible. 
\begin{rmk}
Notice that we have included the term $\frac{Q}{2}\ln g$ in the expression of the measure: it corresponds to the quantum corrections that have to be added so that the expression defined above actually corresponds to the GMC measure of the field $X_{g}$. Indeed if we consider as background metric the round metric $\hat{g}$ introduced in~\eqref{eq:round_metric} we have that:
\[
\lim\limits_{\varepsilon\rightarrow0}M_{b,\hat g}^{\varepsilon}(\mathrm d x)=e^{db^2\tilde{C}_{\hat{g}}}\lim\limits_{\varepsilon\rightarrow0}e^{dbX_{\hat{g},\varepsilon}(x)-\frac{(db)^2}2\expect{X_{\hat{g},\varepsilon}(x)^2}}\mathrm{d}\lambda_{\hat{g}}(x)
\]
since an explicit computation yields the asymptotic behaviour 
\[
\expect{X_{\hat{g},\varepsilon}(x)^2}=\frac{2}{d}\left(\ln\frac{1}{\varepsilon}-\frac12\hat{g}(x)+\Tilde{C}_{\hat{g}}\right)+o(1)
\]
where $\Tilde{C}_{\hat{g}}$ is some positive constant whose value is not relevant in our context. This is also another way to see that the constant $Q$ has to be chosen equal to $b+\frac{1}{b}$ in order to take into account the quantum corrections.
\end{rmk}
\begin{rmk}
To conform to the usual conventions of the literature, one may choose instead to work with a pair $(\gamma,Q_{\gamma})$ where we have set
\[
\gamma= b\sqrt{2d}\quad\text{and}\quad Q_{\gamma}=\frac{\gamma}2+\frac d{\gamma}.
\]
By doing so we are lead to working with
\begin{align*}
    \expect{X_0(x)X_0(y)}=G_{\hat{g}}(x,y)\\
   \mathrm d M_{\gamma, \hat g}:=\lim\limits_{\varepsilon\rightarrow0}\varepsilon^{\frac{\gamma^2}{2}}e^{\gamma(X_0+\frac{Q_{\gamma}}{2}\ln \hat{g})}\mathrm{d}\lambda
\end{align*} the latter being well defined provided that $\gamma<\sqrt{2d}$. A table of the correspondence between these different conventions can be found in Appendix~\ref{table_phy/maths}.
\end{rmk}

To see that the introduction of these objects is relevant in the context of conformal geometry we provide a property of conformal covariance under M\"obius transforms:
\begin{prop}\label{GMC_Mobius}
Let $F$ be any bounded continuous function on $H^{-\frac d2}(\mathbb{R}^d,\hat g)$ and $f$ in $C_c^{\infty}(\mathbb{R}^d)$. Then for any M\"obius transform $\psi$ the following equality in law holds:
\begin{equation*}
    \left(F(X_{\hat{g}}),\int_{\mathbb{R}^d}f\mathrm d M_{b,\hat g}\right)\stackrel{(law)}=\left(F(X_{\hat{g}}\circ\psi^{-1}-m_{\hat{g}_{\psi}}(X_{\hat{g}})),e^{-db m_{\hat{g}_{\psi}}(X_{\hat{g}})}\int_{\mathbb{R}^d}f\circ\psi e^{db\frac Q2\ln\frac{\hat{g}_{\psi}}{\hat{g}}}\mathrm d M_{b,\hat g}\right)
\end{equation*}
where we have denoted $\hat g_{\psi}=\norm{\psi'}^2\hat g\circ\psi$.
\end{prop}
\begin{proof}
We come back to the $\varepsilon$-regularization of the field to get that
\begin{align*}
    \int_{\mathbb{R}^d}f\varepsilon^{db^2}e^{db(X_{\hat{g},\epsilon}+\frac Q2\ln\hat{g})}\mathrm{d}\lambda&=\int_{\mathbb{R}^d}f\circ\psi\varepsilon^{db^2}e^{db(X_{\hat{g},\epsilon}\circ\psi+\frac Q2\ln\hat{g}\circ\psi)}\norm{\psi'}^d\mathrm{d}\lambda\\
    &=\int_{\mathbb{R}^d}f\circ\psi e^{db\frac Q2\ln\frac{\hat{g}_{\psi}}{\hat{g}}}\left(\frac{\varepsilon}{\norm{\psi'}}\right)^{db^2}e^{db(X_{\hat{g},\epsilon}\circ\psi+\frac Q2\ln\hat{g})}\mathrm{d}\lambda.
\end{align*}
Next we observe that the GMC measures defined by $\lim\limits_{\epsilon\rightarrow0}\left(\frac{\varepsilon}{\norm{\psi'}}\right)^{db^2}e^{db(X_{\hat{g},\epsilon}\circ\psi+\frac Q2\ln\hat{g})}\mathrm{d}\lambda$ and $\lim\limits_{\varepsilon\rightarrow0}\varepsilon^{db^2}e^{db((X_{\hat{g}}\circ\psi)_{\varepsilon}+\frac Q2\ln\hat{g})}\mathrm{d}\lambda$ converge actually in probability to the same limiting random measure, since away from the point mapped to $\infty$ (say on $\mathbb{R}^d\setminus B(\psi^{-1}(\infty),\delta)$) one has that $\lim\limits_{\varepsilon\rightarrow0} \expect{(X_{\hat{g},\epsilon}\circ\psi)^2}-\expect{(X_{\hat{g}}\circ\psi)_{\frac{\varepsilon}{\norm{\psi'}}}^2}=0$ while close to $\psi^{-1}(\infty)$ the mass becomes negligible:
\[
    \expect{\int_{B(\psi^{-1}(\infty),\delta)}\left(\frac{\varepsilon}{\norm{\psi'}}\right)^{db^2}e^{db(X_{\hat{g},\epsilon}\circ\psi+\frac Q2\ln\hat{g})}\mathrm{d}\lambda}\leq C \int_{B(\psi^{-1}(\infty),\delta)} \left(\frac{\hat{g}}{\hat{g}_{\psi}}\right)^{\frac{db^2}{2}}\mathrm{d}\lambda_{\hat{g}}\rightarrow 0
\]
as $\delta\rightarrow0$.
To conclude, we use that $X_{\hat{g}}- m_{\hat{g}_{\psi}}(X_{\hat{g}})$ is distributed like $X_{\hat{g}}\circ\psi$ which follows from Lemma~\ref{lemma:mobius_GFF}.
\end{proof}

\subsection{Probabilistic definition of the path integral}
\subsubsection{Probabilistic interpretation of the path integral}
According to what has just been done, we can now give a meaning to the expression that appears in the definition of the Liouville field~\eqref{Liouville_field}. Indeed, for $F:H^{-\frac d2}(\mathbb{R}^d,\hat g)\to\mathbb{R}$ we may interpret the term 
\[
    \frac{1}{\mathcal{Z}_g}\int F(X)e^{-S_L(X,g)}D_gX
\]
by using the tools introduced above; this leads us to considering the following expression for it:
\begin{equation}\label{eq:partition_function}
\begin{split}
\Pi_{b,\mu}(g,F):=\frac 1{\mathcal{Z}_g}\int_{\mathbb{R}}& \mathbb{E}\Big[  F\left(X_g+\frac Q2\ln g +c\right)\\
&\exp\left(-\frac{d}{4\gamma_d}\int_{\mathbb{R}^d} 2Q\mathcal{Q}_g (X_g+c) \mathrm{d}\lambda_g-\mu\int_{\mathbb{R}^d} e^{db (X_g +\frac{Q}{2}\ln g+c)}\mathrm{d}\lambda\right)\Big]dc
\end{split}
\end{equation}
where we have considered the log-correlated field $X_g$ whose covariance kernel is given by $G_g$, and where the term $e^{db (X_g +\frac{Q}{2}\ln g+c)}\mathrm{d}\lambda$ corresponds to the GMC measure associated to the field $X_g +\frac{Q}{2}\ln g+c$ in the sense of Proposition~\ref{prop:GMC}. 

Here the renormalization factor $\mathcal{Z}_g$ which appears in front of the integral (formally) stands for the total mass of the measure of the log-correlated field and is usually referred to as the \textit{Polyakov-Alvarez conformal anomaly} (\cite{Pol81}, \cite{Alv83}).
Since we work with a Gaussian measure, we may interpret this factor $\mathcal{Z}_g$ as $\left(\det\mathcal{P}_g\right)^{-1/2}$ where $\mathcal{P}_g$ is the GJMS operator introduced above. The latter is known only in low dimensions (see~\cite{BO91} for the four-dimensional case); nonetheless it is conjectured that its variation under a conformal change of metric is given by~\cite[Equation (5.9)]{Di08}:
\begin{equation}\label{alvarez_polyakov}
    -\ln\frac{\det\mathcal{P}_{e^{2\varphi}g}}{\det\mathcal{P}_g}=2 c_{d}\int_{\mathbb{R}^d}\varphi\left(\mathcal{Q}_g+\frac 12\mathcal{P}_g\varphi\right)\mathrm{d}\lambda_g+\int_{\mathbb{R}^d}F_\varphi \mathrm{d}\lambda_{e^{2\varphi} g}-\int_{\mathbb{R}^d}F\mathrm{d}\lambda_g
\end{equation}
where the constant $c_d$ is given by 
\[
c_d:=\frac{1}{\gamma_d}\frac{(-1)^{\frac d2}}{d!}\int_{0}^{\frac d2}\prod_{k=0}^{\frac d2-1}(k^2-t^2)dt
\]
and the terms $F$, $F_\varphi$ that appear on the second integral are additional local curvature invariant terms, which are lower order terms arising from the holographic formula for the $\mathcal{Q}$-curvature \cite{GJ07}.
The first part in Equation~\eqref{alvarez_polyakov} is the \textit{universal part in  Polyakov formulas} and is the piece of the formula that fully makes sense in our context. It is usually referred to as the universal type A anomaly of $-\frac12\ln\det\mathcal{P}$.

As a consequence, we may choose to take only into account the first curvature terms by dropping higher order terms and setting:
\begin{equation}
\mathcal{Z}_g:=\exp\left(-c_d \int_{\mathbb{R}^d}\varphi\left(\mathcal{Q}_{\hat g}+\frac 12\mathcal{P}_{\hat g}\varphi\right)\mathrm{d}\lambda_{\hat g}\right)
\end{equation}
for $g=e^{2\varphi}\hat g$ a metric conformally equivalent to the round one $\hat g$ (in particular $Z(\hat g)=1$). We stress that our definition of $\mathcal{Z}_g$ is purely a convention and that its link with the regularized determinant is purely conjectural for the time being in the case $d\geq6$.

Therefore all the terms in Equation~\eqref{eq:partition_function} make sense; however the quantity $\Pi_{b,\mu}(g,1)$, even if it has a meaning, is actually ill-defined. Indeed, consider the setting where $g=\hat{g}$ the round metric, so that one has 
\[
\Pi_{b,\mu}(\hat{g},1)=\int_{\mathbb{R}}\expect{e^{-dQc-\mu e^{db c}M_{b,\hat g}(\mathbb{R}^d)}}dc.
\]
Then the integral in the variable $c$ is divergent, because of the behaviour as $c\rightarrow-\infty$ of the $e^{-dQc}$ term in the integrand. The divergence of the partition function is the quantum interpretation of the fact that a conformal metric with constant negative $\mathcal{Q}$-curvature on the sphere must have logarithmic singularities. In order to get rid of this issue, we will need to consider the Liouville action to which we have added conical singularities, which corresponds in the language of quantum field theory to inserting so-called \textit{Vertex Operators}.
Formally, for $x$ in $\mathbb{R}^d$ and a weight $\alpha\in\mathbb{R}$, these are defined by setting
\begin{equation}
V_{\alpha}(x)=e^{d\alpha\phi(x)}
\end{equation}
where $\phi$ is the Liouville field. The expectation of a product of Vertex Operators with respect to the Liouville field is then called \textit{correlation function of Vertex Operators} and is formally defined by  associating to any finite set of pairs $(\bm x,\bm\alpha)\coloneqq \Big((x_1,\alpha_1),\cdots,(x_N,\alpha_N)\Big)$ of elements of $\mathbb{R}^d\times\mathbb{R}$ (with $x_1,\cdots,x_N$ distinct) the quantity $<\prod_{i=1}^NV_{\alpha_i}(x_i)>_{b,\mu}$ given by
\begin{align*}
 &\frac{1}{\mathcal{Z}_g}\int_{\mathbb{R}}\expect{\prod_{i=1}^Ne^{d\alpha_i (X_g+\frac Q2\ln g+c)(x_i)}\exp\left(-\frac{d}{4\gamma_d}\int_{\mathbb{R}^d} 2Q\mathcal{Q}_g (X_g+c) \mathrm{d}\lambda_g-\mu\int_{\mathbb{R}^d} e^{db (X_g +\frac{Q}{2}\ln g+c)}\mathrm{d}\lambda\right)}dc.
\end{align*}
This corresponds to adding conical singularities $(\bm x,\bm{\alpha})$ to the Liouville functional of the theory.
\begin{rmk}
To conform to the usual convention in the mathematics literature, one usually works with the Vertex Operators being defined by the expression
\[
V_{\alpha}(x)=e^{\alpha\sqrt{\frac d2}\phi(x)}=e^{\alpha \phi_0(x)}
\]
where the field $\phi_0$ corresponds to $\sqrt{\frac{d}{2}}\left(X_g+\frac Q2\ln g+c\right)$ whose covariance is normalised to have a singularity $\sim\ln\frac{1}{\norm{x-y}}$ on the diagonal.
\end{rmk}
However the latter writing is purely heuristic since the field $X_g$ is not regular and cannot be evaluated pointwise. However we will see that working with a log-correlated field will enable us to provide a rigorous meaning to it. In a similar way as the one thanks to which we have defined the GMC measure this will be done by considering a limiting procedure involving a regularization of the field:
\begin{defi}
For positive $\varepsilon$, consider $X_{g,\varepsilon}$ to be the spherical average regularization of the field $X_g$ defined in Equation~\eqref{eq:average}. We define the $\varepsilon$-partition function associated to the marked points $(\bm{x},\bm{\alpha})$ by setting
\begin{equation}\label{eq:reg_partition_function}
\begin{split}
\Pi_{b,\mu}^{(\bm{x},\bm{\alpha})}(g,F;\varepsilon)\coloneqq &\frac{1}{\mathcal{Z}_g}\int_{\mathbb{R}}\mathbb{E}\left[F\left(X_{g,\varepsilon}+\frac Q2\ln g +c\right)\right.\left.\prod_{i=1}^N\varepsilon^{d\alpha_i^2}e^{d\alpha_i (X_{g,\varepsilon}+\frac Q2\ln g+c)(x_i)}\right.\\
     &\left.\exp\left(-\frac{dQ}{2\gamma_d}\int_{\mathbb{R}^d} \mathcal{Q}_g(X_{g,\varepsilon}+c) \mathrm{d}\lambda_g-\mu\int_{\mathbb{R}^d} \varepsilon^{db^2}e^{db (X_{g,\varepsilon} +\frac{Q}{2}\ln g+c)}\mathrm{d}\lambda\right)\right]dc .
\end{split}
\end{equation}The regularized correlation function of Vertex Operators is defined as
\begin{equation}
<\prod_{i=1}^NV_{\alpha_i}(x_i)>_{b,\mu,\varepsilon}\coloneqq \Pi_{b,\mu}^{(\bm{x},\bm{\alpha})}(g,1;\varepsilon).
\end{equation}
\end{defi}
In the sequel, we wish to inquire under which assumptions the limits \[
\Pi_{b,\mu}^{(\bm{x},\bm{\alpha})}(g,F)\coloneqq\lim\limits_{\varepsilon\rightarrow0}\Pi_{\gamma,\mu}^{(\bm{x},\bm{\alpha})}(g,F;\varepsilon)
\] and
\begin{equation}
<\prod_{i=1}^NV_{\alpha_i}(x_i)>_{b,\mu}\coloneqq\lim\limits_{\varepsilon\rightarrow0}<\prod_{i=1}^NV_{\alpha_i}(x_i)>_{b,\mu,\varepsilon}
\end{equation}
do exist.

\subsubsection{Existence of the path integral}
First of all, let us note that the quantity $\int_{\mathbb{R^d}}\mathcal{Q}_g\mathrm{d}\lambda_g$ does not depend on the metric $g$ in the conformal class of the round metric $\hat{g}$ and its value is therefore given by $(d-1)!\norm{\mathbb{S}^d}=2\gamma_d$. As a consequence we have that for $g$ conformally equivalent to $\hat g$
\begin{align*}
&\mathcal{Z}_g\Pi_{b,\mu}^{(\bm{x},\bm{\alpha})}(g,F;\varepsilon)=\prod_{i=1}^N g(x_i)^{\frac{d\alpha_i}2Q}\int_{\mathbb{R}}e^{dc(\sum_i\alpha_i-Q)}\\
&\mathbb{E}\left[F\left(X_{g,\varepsilon}+\frac Q2\ln g +c\right)\prod_{i=1}^N\varepsilon^{d\alpha_i^2}e^{d\alpha_i X_{g,\varepsilon}(x_i)}\exp\left(-\frac{dQ}{2\gamma_d}\int_{\mathbb{R}^d} \mathcal{Q}_gX_{g,\varepsilon} \mathrm{d}\lambda_g-\mu\int_{\mathbb{R}^d} \varepsilon^{db^2}e^{db (X_{g,\varepsilon} +\frac{Q}{2}\ln g+c)}\mathrm{d}\lambda\right)\right]dc  .
\end{align*}

At first, let us assume that $g$ is the round metric $\hat{g}$ (the statement of Theorem~\ref{thm:A_anomaly} shows that this no restriction); then the term $\int_{\mathbb{R}^d} \mathcal{Q}_gX_{g} \mathrm{d}\lambda_g$ vanishes, and since we know that $\expect{X_{\hat{g}}(x)^2}=\frac2d\left(\ln\frac{1}{\varepsilon}-\frac12\hat{g}(x)+\Tilde{C}_{\hat{g}}\right)+o(1)$ we may rewrite $\Pi_{b,\mu}^{(\bm{x},\bm{\alpha})}(\hat{g},F;\varepsilon)$ under the form
\begin{align*}
&\prod_{i=1}^Ne^{\Tilde{C}_{\hat{g}}d\alpha_i^2} \hat{g}(x_i)^{\frac{d\alpha_i}2(Q-\alpha_i)}\int_{\mathbb{R}}e^{cd(\sum_i\alpha_i-Q)}(1+o(1))\\
&\mathbb{E}\left[F\left(X_{\hat g,\varepsilon}+\frac Q2\ln \hat g +c\right)\prod_{i=1}^Ne^{d\alpha_i X_{\hat g,\varepsilon}(x_i)-\frac{(d\alpha_i)^2}2\expect{X_{\hat g,\varepsilon}(x_i)^2}}\exp\left(-\mu e^{db c}\int_{\mathbb{R}^d} \varepsilon^{db^2}e^{db (X_{\hat g,\varepsilon} +\frac{Q}{2}\ln \hat{g})}\mathrm{d}\lambda\right)\right]dc
\end{align*}
where the quantity denoted by $o(1)$ is purely deterministic and converges towards zero as $\varepsilon$ goes to zero.

The exponential terms that appear can be interpreted as Girsanov transforms (see Theorem~\ref{Girsanov}): working under the probability measure whose Radon-Nikodym derivative with respect to the measure of the log-correlated field $X_{\hat{g}}$ is given by
\[
\prod_{i=1}^Ne^{d\alpha_i X_{\hat g,\varepsilon}(x_i)-\frac{(d\alpha_i)^2}2\expect{X_{\hat g,\varepsilon}(x_i)^2}}
\prod_{i\neq j} e^{\frac{d\alpha_id\alpha_j}2 \expect{X_{\hat{g},\varepsilon}(x_i)X_{\hat{g},\varepsilon}(x_j)}}
\]
is tantamount to shifting the law of the field by an additive factor of 
\begin{equation}
H_{\hat{g},\varepsilon}(x):=\sum_{i=1}^N 2\alpha_i G_{\hat{g},\varepsilon}(x,x_i),
\end{equation}
where we have denoted by $G_{\hat{g},\varepsilon}(x,y)$ the $\varepsilon$-spherical average regularization of $G_{\hat{g}}$, defined by setting
\begin{equation}
G_{\hat{g},\varepsilon}(x,y)\coloneqq \frac{1}{\norm{\mathbb{S}^{d-1}}^2}\iint_{(\mathbb{S}^{d-1})^2} G_{\hat g}(x+\varepsilon z_1,y+\varepsilon z_2)\mathrm{d}\lambda_{\partial}(z_1)\mathrm{d}\lambda_{\partial}(z_2).
\end{equation}
Therefore,  $\Pi_{b,\mu}^{(\bm{x},\bm{\alpha})}(\hat{g},F)$ is actually given by
\begin{align*}
e^{C_{\hat g}(\bm{x},\bm{\alpha})}\prod_{i=1}^N \hat g(x_i)^{\frac{d\alpha_i}2(Q-\alpha_i)} &\lim\limits_{\varepsilon\rightarrow0}\int_{\mathbb{R}}e^{cd(\sum_i\alpha_i-Q)}\\
&\mathbb{E}\left[F\left(X_{\hat{g},\varepsilon}+H_{\hat{g},\varepsilon}+\frac Q2\ln \hat g +c\right)\exp\left(-\mu e^{db c}\int_{\mathbb{R}^d} \varepsilon^{db^2}e^{db (X_{\hat{g},\varepsilon} +H_{\hat{g},\varepsilon}+\frac{Q}{2}\ln \hat{g})}\mathrm{d}\lambda\right)\right]dc
\end{align*}
where we have set $C_{\hat g}(\bm{x},\bm{\alpha})=d\sum_{i\neq j}\alpha_i\alpha_jG_{\hat{g}}(x_i,x_j)$, and provided that the limit exists.

As a consequence the convergence of the $\varepsilon$-partition function in the round metric $\Pi_{b,\mu}^{(\bm{x},\bm{\alpha})}(\hat{g},F;\varepsilon)$ is ensured by a regularity result for the GMC measure (which follows from \cite[Lemma 2.7]{RV16}):
\begin{lemma}\label{lemma_conv}
For positive $\varepsilon$, denote by $Z_{\varepsilon,\hat g}$ the random variable 
\[
 Z_{\varepsilon,\hat g}:=\int_{\mathbb{R}^d} \varepsilon^{db^2}e^{db (X_{\hat{g},\varepsilon} +H_{\hat{g},\varepsilon}+\frac{Q}{2}\ln \hat{g})}\mathrm{d}\lambda
\]
Then for any negative $s$:
\begin{itemize}
	\item If for any $i$, $\alpha_i<\frac Q2$, 
\[
\lim\limits_{\varepsilon\rightarrow0} \expect{Z_{\varepsilon,\hat g}^s}=\expect{Z_{\hat g}^s}
\]
where $Z_{\hat g}\coloneqq\int_{\mathbb{R}^d}e^{db H_{\hat{g}}}dM_{b,\hat g}$ satisfies $0<\expect{Z_{\hat g}^s}<\infty$.
	\item If for some $i$, $\alpha_i\geq \frac Q2$, then
	\[
\lim\limits_{\varepsilon\rightarrow0} \expect{Z_{\varepsilon, \hat g}^s}=0
\]
\end{itemize}
\end{lemma}

With this result at hand, we are in position to provide a rigorous statement for the convergence of the partition function. This extends the construction proposed in~\cite{DKRV16}.
\begin{theorem}\label{partition_limit}
Let $F$ be any continuous bounded functional over the space $H^{-\frac d2}(\mathbb{R}^d,\hat g)$ and let $(\bm x,\bm \alpha)$ be finitely many marked points with $x_1,\cdots,x_N$ distinct.
Assume that the bound
\begin{equation}
\sum_{i=1}^N\alpha_i-Q>0
\end{equation}
holds. Then the limit 
\[
\Pi_{b,\mu}^{(\bm{x},\bm{\alpha})}(\hat{g},F):=\lim\limits_{\varepsilon\rightarrow0}\Pi_{b,\mu}^{(\bm{x},\bm{\alpha})}(\hat g,F;\varepsilon)
\]
exists and is non-zero if and only if $\alpha_i<\frac Q2$ for all $1\leq i\leq N$. Moreover 
\begin{equation}\label{Equation_partition}
\begin{split}
\Pi_{b,\mu}^{(\bm{x},\bm{\alpha})}(\hat{g},F)=
e^{C_{\hat g}(\bm{x},\bm{\alpha})}&\prod_{i=1}^N \hat g(x_i)^{\frac{\Delta_{\alpha_i}}{2}} \times\\
&\int_{\mathbb{R}}e^{cd(\sum_i\alpha_i-Q)}\mathbb{E}\left[F\left(X_{\hat{g}}+H_{\hat{g}}+\frac Q2\ln g +c\right)e^{-\mu e^{db c}Z_{\hat g}}\right]dc
\end{split}
\end{equation}
where $\Delta_{\alpha}:=d\alpha(Q-\alpha)$ is the dimension of the vertex operator $V_{\alpha}(x)$, and
\begin{equation}
H_{\hat{g}}(x):=\sum_{i=1}^N 2\alpha_i G_{\hat{g}}(x,x_i),
\end{equation}
\end{theorem}
The bounds that appear in the above statement are usually referred to as the \textit{Seiberg bounds} in the two-dimensional setting~\cite{Sie90}. They correspond to the quantum version of the bounds required to ensure the existence of a conformal structure with constant negative $\mathcal{Q}$-curvature on the sphere in Theorem~\ref{classical_conical}.

The convergence of the partition function shows that the law of the Liouville Field can indeed been defined in a meaningful way. In the next subsection we review some properties enjoyed by the quantum field theory thus defined.

\subsection{First properties}
\subsubsection{Conformal change of metrics: A-Type anomaly}
In the previous paragraph, we have defined the quantities $\Pi_{b,\mu}^{(\bm{x},\bm{\alpha})}(g,F)$ in the special case where we have considered as background metric $g=\hat{g}$. However, this is absolutely no restriction since the latter is independent of the background metric $g$ conformal to $\hat{g}$ up to a multiplicative factor given by the \textit{A-type conformal anomaly coefficient}. This proposition is natural if we recall that the theory was aimed at describing canonical conformal structures on the sphere and therefore should not depend on the background metric initially considered.

We adopt the conventions that $e^{2\varphi}g$ is conformally equivalent to $g$ when $\varphi-m_g(\varphi)$ belongs to the Hilbert space $H^{\frac d2}(\mathbb{R}^d,g)$. The following statement matches the one derived in \cite[III.C]{LO18}:
\begin{theorem}[The A-type anomaly]\label{thm:A_anomaly}
Let $g=e^{2\varphi}\hat{g}$ be a metric conformally equivalent to $\hat{g}$. Then, under the assumptions of Theorem~\ref{partition_limit}, $\lim\limits_{\varepsilon\rightarrow0}\Pi_{b,\mu}^{(\bm{x},\bm{\alpha})}(g,F;\varepsilon)$ exists and is positive. Moreover, one has the following \textit{A-Type anomaly}:
\begin{equation}
\ln\frac{\Pi_{b,\mu}^{(\bm{x},\bm{\alpha})}(g,F)}{\Pi_{b,\mu}^{(\bm{x},\bm{\alpha})}(\hat{g},F)}=d(-1)^{\frac d2}a\left(\int_{\mathbb{R}^d} 2\varphi\left(\mathcal{Q}_{\hat{g}}+\mathcal{P}_{\hat g}\varphi\right) \mathrm{d}\lambda_{\hat{g}}\right)
\end{equation}
where \begin{equation}
a\coloneqq\frac{2}{(d!)^2\norm{\mathbb S^d}}\int_{0}^{\frac d2}\prod_{k=0}^{\frac d2-1}(k^2-t^2)dt + \frac{(-1)^{\frac d2}}{(d-1)!\norm{\mathbb S^d}}Q^2
\end{equation}
is the so-called \textit{A-type conformal anomaly coefficient}.
\end{theorem}
In the two-dimensional setting, this statement is usually referred to as the \textit{Weyl anomaly}~\cite[Theorem 3.11]{DKRV16} and the quantity $c:=1+6Q^2=-24\pi a$ corresponds to the \textit{central charge}. The first term in the expression of $a$, which comes from the definition of the partition function $\mathcal{Z}_g$, is a convention but is conjectured to be related to a regularized determinant for $d\geq 6$ as explained around Equation~\eqref{alvarez_polyakov}. 
\begin{proof}
Let us consider the expression of $\Pi_{b,\mu}^{(\bm{x},\bm{\alpha})}(g,F;\varepsilon)$ without the renormalization factor $\mathcal{Z}_g$:
\begin{align*}
&\mathcal{Z}_g\Pi_{b,\mu}^{(\bm{x},\bm{\alpha})}(g,F;\varepsilon)=\prod_{i=1}^N g(x_i)^{\frac{d\alpha_i}2Q}\int_{\mathbb{R}}e^{dc(\sum_i\alpha_i-Q)}\\
&\mathbb{E}\left[F\left(X_{g}+\frac Q2\ln g +c\right)\prod_{i=1}^N\varepsilon^{d\alpha_i^2}e^{d\alpha_i X_{g,\varepsilon}(x_i)}\exp\left(-\frac{dQ}{2\gamma_d}\int_{\mathbb{R}^d} \mathcal{Q}_gX_{g} \mathrm{d}\lambda_g-\mu\int_{\mathbb{R}^d} \varepsilon^{db^2}e^{db (X_{g,\varepsilon} +\frac{Q}{2}\ln g+c)}\mathrm{d}\lambda\right)\right]dc  
\end{align*}
First of all, we use the fact that $X_g$ has same law as $X_{\hat{g}}-m_g(X_{\hat{g}})$: by using the change of variable $c\leftrightarrow c-m_g(X_{\hat{g}})$ we get the same expression as above but instead of $X_g$ we work with $X_{\hat{g}}$.

Next, we can write that 
\[
-\frac{dQ}{2\gamma_d}\int_{\mathbb{R}^d} \mathcal{Q}_gX_{\hat{g}} \mathrm{d}\lambda_g=-\frac{dQ}{2\gamma_d}\int_{\mathbb{R}^d} (\mathcal{Q}_{\hat{g}}+\mathcal{P}_{\hat{g}}\varphi)X_{\hat{g}} \mathrm{d}\lambda_{\hat g}=-Q(X,\mathcal{P}_{\hat{g}}\varphi-m_{\hat{g}}(\varphi))_{\hat g}\]
As a consequence (again using a Girsanov transform) working under the weighted measure whose Radon-Nikodym derivative with respect to the one of $X_{\hat{g}}$ is given by
\[
\exp\left(-\frac{dQ}{2\gamma_d}\int_{\mathbb{R}^d} \mathcal{Q}_gX_{\hat{g}} \mathrm{d}\lambda_g\right)\exp\left(-\frac{Q^2}2(\varphi,\mathcal{P}_{\hat g}\varphi)_{\hat g}\right)
\]
is tantamount to shifting the law of $X_{\hat{g}}$ by an additive factor of $- Q\left(\varphi-m_{\hat{g}}(\varphi)\right)$. In particular
\begin{align*}
&\mathcal{Z}_g\Pi_{b,\mu}^{(\bm{x},\bm{\alpha})}(g,F;\varepsilon)=\exp\left[\frac{Q^2}{2}\Big((\varphi,\mathcal{P}_{\hat g}\varphi)_{\hat g}+dm_{\hat{g}}(\varphi)\Big)\right]\prod_{i=1}^N \hat{g}(x_i)^{\frac{d\alpha_i}2Q}\int_{\mathbb{R}}e^{d(\sum_i\alpha_i-Q)(c+\frac Q2m_{\hat{g}}(\ln\frac{g}{\hat{g}}))}\\
&\mathbb{E}\left[F\left(X_{\hat{g}}+\frac Q2\ln \hat{g} +c+\frac Q2m_{\hat{g}}(\ln\frac{g}{\hat{g}})\right)\prod_{i=1}^N\varepsilon^{d\alpha_i^2}e^{d\alpha_i X_{\hat{g},\varepsilon}(x_i)}e^{-\mu\int_{\mathbb{R}^d} \varepsilon^{db^2}e^{db (X_{\hat{g},\varepsilon} +\frac{Q}{2}\ln \hat{g}+c+\frac Q2m_{\hat{g}}(\ln\frac{g}{\hat{g}}))}\mathrm{d}\lambda}\right]dc(1+o(1))
\end{align*}
where the deterministic $o(1)$ absorbs the fact that we work with a regularization of the quantities involved.
The change of variable $c\leftrightarrow c+Qm_{\hat{g}}(\varphi)$ combined with the fact that \[
dm_{\hat{g}}(\varphi)=\frac{d}{2\gamma_d}\int_{\mathbb{R}^d} \mathcal{Q}_{\hat{g}} \varphi \mathrm{d}\lambda_{\hat{g}}
\]
provides the coefficient in front of $Q^2$ in the statement of our claim.

The part independent of $Q$ is given by the contribution of the renormalization factor $\mathcal{Z}_g$.
\end{proof}

\subsubsection{KPZ scaling laws and KPZ formula for the Vertex Operators}
Now that we have given a proper meaning to the correlation function defined in terms of the partition function $\Pi_{b,\mu}^{(\bm{x},\bm{\alpha})}(g,1)$, we are interested in its first properties, and more precisely we wish to understand its dependence in the cosmological constant $\mu$ as well as its behaviour under conformal transformations. The $\mu$-dependence is usually referred to as the Knizhnik-Polyakov-Zamoldchikov (KPZ) scaling law discovered in~\cite{KPZ}. The results that we find agree with the ones that can be found in the physics literature (\cite[Equation (22)]{LO18}), and are consistent with the two-dimensional case~\cite[Theorem 3.5]{DKRV16}:
\begin{theorem}[KPZ scaling laws and KPZ formula]\label{KPZ}
Assume that the marked points $(\bm{x},\bm{\alpha})$ satisfy the assumptions of Theorem~\ref{partition_limit}.

The correlation function obeys the following dependence in the cosmological constant $\mu$:
\begin{equation}
<\prod_{i=1}^NV_{\alpha_i}(x_i)>_{b,\mu}=\mu^{-\frac{\sum_{i=1}^N\alpha_i-Q}{b}}<\prod_{i=1}^NV_{\alpha_i}(x_i)>_{b,1}.
\end{equation}
Moreover, the Vertex Operators are primary operators of dimension $\Delta_{\alpha}=d\alpha(Q-\alpha)$ in the sense that they satisfy the following property of covariance under M\"obius transforms $\psi$:
\begin{equation}
<\prod_{i=1}^NV_{\alpha_i}(\psi(x_i))>_{b,\mu}=\prod_{i=1}^N\norm{\psi'(x_i)}^{-\Delta_{\alpha_i}}<\prod_{i=1}^NV_{\alpha_i}(x_i)>_{b,\mu}.
\end{equation}
\end{theorem}
\begin{proof}
The first fact follows immediately from the following simple expression for the correlation function
\begin{equation}
<\prod_{i=1}^NV_{\alpha_i}(x_i)>_{b,\mu}=e^{C(\bm{x},\bm{\alpha})}\prod_{i=1}^N \hat{g}(x_i)^{\frac12\Delta_{\alpha_i}}\frac{\Gamma(\frac{\sum_{i=1}^N\alpha_i-Q}{b})}{db}\mu^{\frac{Q-\sum_{i=1}^N\alpha_i}{b}}\expect{Z_{\hat g}^{-\frac{\sum_{i=1}^N\alpha_i-Q}{b}}}.
\end{equation} 
This expression is derived by performing the change of variable $c\leftrightarrow \mu e^{dbc}$ in  and exchanging integral and expectation (by Fubini-Tonelli theorem since all the quantities are positive) in Equation~\eqref{eq:reg_partition_function}.

For the second point, we actually show a more general result which we formulate as:
\begin{equation}\label{eq:cov_par_fun}
\Pi^{(\psi(\bm{x}),\bm{\alpha})}_{b,\mu}(\hat{g},F)=\prod_{i=1}^N\norm{\psi'(x_i)}^{-\Delta_{\alpha_i}}\Pi^{(\bm{x},\bm{\alpha})}_{b,\mu}\left(\hat{g},F(\cdot\circ\psi^{-1}+Q\ln\norm{(\psi^{-1})'})\right)
\end{equation}
where we have set $\psi(\bm x)\coloneqq(\psi(x_1),\cdots,\psi(x_N))$. To start with, Proposition~\ref{GMC_Mobius} allows to rewrite the expectation in Equation~\eqref{Equation_partition} as
\begin{align*}
&\mathbb{E}\left[F\left((X_{\hat{g}}+H_{\hat{g},\psi}\circ\psi+\frac Q2\ln \hat{g}\circ\psi)\circ\psi^{-1} +c-m_{\hat{g}_{\psi}}(X_{\hat{g}})\right)\exp\left(-\mu e^{db(c-m_{\hat{g}_{\psi}}(X_{\hat{g}}))}\int_{\mathbb{R}^d} e^{db (H_{\hat{g},\psi}\circ\psi +\frac{Q}{2}\varphi)}\mathrm dM_{b,\hat g}\right)\right]
\end{align*}
where we have denoted $\varphi=\ln\frac{g_{\psi}}{g}$ and $H_{\hat{g},\psi}=\sum_{i=1}^N2\alpha_iG_{\hat{g}}(\cdot,\psi(x_i))$. Now by Lemma~\ref{Mobius_Green} the latter is equal to
\begin{align*}
\mathbb{E}\left[F\left((X_{\hat{g}}+H_{\hat{g}}-\sum_{i=1}^N\frac{\alpha_i}{2}\varphi+\frac Q2\ln \hat{g}\circ\psi)\circ\psi^{-1} +c-m_{\hat{g}_{\psi}}(X_{\hat{g}})-\sum_{i=1}^N\frac{\alpha_i}{2}\varphi(x_i)\right)\right.\\
\left.\exp\left(-\mu e^{db\left(c-m_{\hat{g}_{\psi}}(X_{\hat{g}})-\sum_{i=1}^N\frac{\alpha_i}{2}\varphi(x_i)\right)}\int_{\mathbb{R}^d} e^{db (H_{\hat{g}} +\frac{Q-\sum_i\alpha_i}{2}\varphi)}\mathrm dM_{b,\hat g}\right)\right].
\end{align*}
Then, in the expression of $\Pi^{(\bm{x},\bm{\alpha})}_{b,\mu}(\hat{g},F)$ we can make the change of variable $c\leftrightarrow c -m_{\hat{g}_{\psi}}(X_{\hat{g}})-\sum_{i=1}^N\frac{\alpha_i}{2}\varphi(x_i)$ to get that
\begin{align*}
    &\Pi^{(\psi(\bm{x}),\bm{\alpha})}_{b,\mu}(\hat{g},F)=e^{C_{\hat g}(\psi(\bm{x}),\bm{\alpha})}\prod_{i=1}^N\hat{g}(\psi(x_i))^{\frac{\Delta_{\alpha_i}}2}\int_{\mathbb{R}}e^{ds\left(c+\sum_{i=1}^N\frac{\alpha_i}{2}\varphi(x_i)\right)}\\
    &\mathbb{E}\left[e^{dsm_{\hat{g}_{\psi}}(X_{\hat{g}})}F\left((X_{\hat{g}}+H_{\hat{g}}-\sum_{i=1}^N\frac{\alpha_i}{2}\varphi+\frac Q2\ln \hat{g}\circ\psi)\circ\psi^{-1} +c\right)\exp\left(-\mu e^{dbc}\int_{\mathbb{R}^d} e^{db (H_{\hat{g}} -\frac{s}{2}\varphi)}dM_{b, \hat g}\right)\right],
\end{align*}
where as before $s=\sum_{i=1}^N\alpha_i-Q$.
The exponential term $e^{dsm_{\hat{g}_{\psi}}(X_{\hat{g}})}$ is a Girsanov transform, since $dm_{\hat{g}_{\psi}}(X_{\hat{g}})=\left(X_{\hat{g}},\mathcal{P}_{\hat{g}}\frac{1}{2}(\varphi-m_{\hat{g}}(\varphi))\right)$ as explained in the proof of Lemma~\ref{Mobius_Green}. This transform has the effect of shifting the law of $X_{\hat{g}}$ by an additive term $\frac{s}{2}(\varphi-m_{\hat{g}}(\varphi))$ and multiplying the whole expectation by $\frac{ds^2}2m_{\hat{g}}(\varphi)$. Therefore the expectation may be rewritten as 
 \begin{align*}
e^{ds\left(\frac{s}2m_{\hat{g}}(\varphi)+\sum_{i=1}^N\frac{\alpha_i}{2}\varphi(x_i)\right)}\mathbb{E}\Big[F&\left((X_{\hat{g}}+H_{\hat{g}}-\frac{Q}{2}\varphi+\frac Q2\ln \hat{g}\circ\psi)\circ\psi^{-1} +c-\frac{s}{2}m_{\hat{g}}(\varphi)\right)\\
&\left.\exp\left(-\mu e^{db(c-\frac{s}{2}m_{\hat{g}}(\varphi))}\int_{\mathbb{R}^d} e^{db H_{\hat{g}}}dM_{b, \hat g}\right)\right].
 \end{align*}
Eventually we again perform a change of variable $c\leftrightarrow c-\frac s2m_{\hat{g}}(\varphi)$ to get \begin{align*}
    &\Pi^{(\psi(\bm{x}),\bm{\alpha})}_{b,\mu}(\hat{g},F)=e^{C_{\hat g}(\psi(\bm{x}),\bm{\alpha})}\prod_{i=1}^N\hat{g}(\psi(x_i))^{\frac{\Delta_{\alpha_i}}2}e^{ds\sum_{i=1}^N\frac{\alpha_i}{2}\varphi(x_i)}\int_{\mathbb{R}}e^{dcs}\\
    &\mathbb{E}\left[F\left((X_{\hat{g}}+H_{\hat{g}}+\frac Q2\ln \hat{g}-\frac Q2\ln\norm{\psi'}^2)\circ\psi^{-1} +c\right)\exp\left(-\mu e^{dbc}\int_{\mathbb{R}^d} e^{db H_{\hat{g}}}dM_{b, \hat g}\right)\right].
\end{align*}
To finish up, notice that thanks to Lemma~\ref{Mobius_Green} 
\[e^{C_{\hat g}(\psi(\bm{x}),\bm{\alpha})}=e^{C_{\hat g}(\bm{x},\bm{\alpha})}e^{- d\sum_{i}\frac{\alpha_i}2\varphi(x_i)\sum_{i\neq j}\alpha_j}=e^{C_{\hat g}(\bm{x},\bm{\alpha})}e^{- ds\sum_{i}\frac{\alpha_i}2\varphi(x_i)+\sum_{i}\frac{\Delta_{\alpha_i}}2\varphi(x_i)},
\]
which implies that 
\begin{align*}
    &\Pi^{(\psi(\bm{x}),\bm{\alpha})}_{b,\mu}(\hat{g},F)=\prod_{i=1}^N\norm{\psi'(x_i)}^{-\Delta_{\alpha_i}}e^{C_{\hat g}(\bm{x},\bm{\alpha})}\int_{\mathbb{R}}e^{dcs}\\
    &\mathbb{E}\left[F\left((X_{\hat{g}}+H_{\hat{g}}+\frac Q2\ln \hat{g})\circ\psi^{-1}+ Q\ln\norm{(\psi^{-1})'} +c\right)\exp\left(-\mu e^{dbc}\int_{\mathbb{R}^d} e^{db H_{\hat{g}}}dM_{b, \hat g}\right)\right].
\end{align*}
\end{proof}

\subsection{Definition of the Liouville field and measure}
Now that we have given a meaning to the path integral the theory, it is possible to make sense of the expression~\eqref{Liouville_field} which defines the law of the Liouville field $\phi$.
\begin{defi}
Consider marked points $(\bm{x},\bm{\alpha})$ satisfying the assumptions of Theorem~\ref{partition_limit}. The Liouville field with marked points $(\bm{x},\bm{\alpha})$ is a random field whose probability law $\mathbb{P}_{b,\mu}^{(\bm{x},\bm{\alpha})}$ is defined by setting for any continuous bounded function $F$ over $H^{-\frac d2}(\mathbb{R}^d,g)$:
\begin{equation}
    \mathbb{E}_{b,\mu}^{(\bm{x},\bm{\alpha})}\left[F(\phi)\right]:=\frac{\Pi_{b,\mu}^{(\bm{x},\bm{\alpha})}(g,F)}{\Pi_{b,\mu}^{(\bm{x},\bm{\alpha})}(g,1)}
\end{equation}
where $g$ is any metric conformal to the round metric $\hat g$ (the A-type anomaly shows that the probability measure is indeed independent of the background metric $g$).
\end{defi}
\begin{defi}
Let $g$ be conformal to $\hat g$. The Liouville measure with marked points $(\bm x,\bm\alpha)$, $Z$ is given by the law of $e^{db\phi}\mathrm{d}\lambda$ where $\phi$ has law $\mathbb{P}_{b,\mu}^{(\bm{x},\bm{\alpha})}$ and where the expression $e^{db\phi}\mathrm{d}\lambda$ should be understood as a GMC measure associated to $\phi$. Put differently, the joint law of the Liouville field and measure is given by
\begin{equation}
    \mathbb{E}_{b,\mu}^{(\bm{x},\bm{\alpha})}\left[F(\phi;\mathrm dZ)\right]:=\frac{\int_{\mathbb{R}}e^{dsc}\mathbb{E}\left[F\left(X_{g}+H_{g}+\frac Q2\ln g +c; e^{db(H_{g}+c)}\mathrm dM_{b, g} \right)\exp\left(-\mu e^{dbc}\int_{\mathbb{R}^d} e^{db H_{g}}\mathrm dM_{b,g}\right)\right]dc}{\int_{\mathbb{R}}e^{dsc}\mathbb{E}\left[\exp\left(-\mu e^{dbc}\int_{\mathbb{R}^d} e^{db H_{g}}\mathrm dM_{b, g}\right)\right]dc}
\end{equation}
where as before $s=\sum_{i=1}^N\alpha_i-Q$, $H_{\hat{g}}=2\sum_{i=1}^N\alpha_iG_{\hat{g}}(\cdot,x_i)$ and $\mathrm dM_{b,\hat g}$ is the GMC measure associated to $X_{g}+\frac Q2\ln g$ as in Proposition~\ref{prop:GMC}. Like before this expression does not depend on the choice of metric $g$ in the conformal class of $\hat g$.
\end{defi}
The theory thus defined is then consistent with what one would expect from LCFT (see \textit{e.g.}~\cite{G96} for the two-dimensional case and~\cite{R16,LO18} for a generalization to higher dimension):
\begin{prop}\label{conformal_covariance}
The random field $\phi$ whose law is given by $\mathbb{P}_{b,\mu}^{(\bm{x},\bm{\alpha})}$ satisfies the following properties:\begin{itemize}
    \item \textbf{Conformal covariance}: For any M\"obius transform $\psi$ of the sphere, the law of $\phi$ under $\mathbb{P}_{b,\mu}^{(\bm{x},\bm{\alpha})}$  is the same as the law of $\phi\circ\psi+Q\ln\norm{\psi'}$ under $\mathbb{P}_{b,\mu}^{(\psi(\bm{x}),\bm{\alpha})}$
    
    \item \textbf{Independence in the background metric}: The law of $\phi$ under $\mathbb{P}_{b,\mu}^{(\bm{x},\bm{\alpha})}$ does not depend on the background metric in the conformal class of the round metric $\hat g$.
    \item \textbf{Dimension of the Vertex Operators}: The Vertex Operators $V_{\alpha}(x)$ are primary fields with conformal weights $\Delta_{\alpha}=d\alpha(Q-\alpha)$.
\end{itemize}
\end{prop}
\begin{proof}
The first and third items follow from the proof of Theorem~\ref{KPZ}, and more precisely from Equation~\eqref{eq:cov_par_fun};  the second item is a straightforward consequence of Theorem~\ref{thm:A_anomaly}.
\end{proof}
Let us comment on these properties. The first one corresponds to the fact that a conformal reparametrisation of the sphere prescribed by a M\"obius transform $\psi$ simply corresponds to a pushforward for the corresponding metric on the sphere. This is due to a well-known property of the Gaussian Multiplicative Chaos measure: for any conformal map $\psi$ the law of the GMC measure defined by exponentiating $X\circ\psi+Q\ln\norm{\psi'}$ is the same as the pushforward of the GMC measure defined by exponentiating $X$ (see Proposition~\ref{GMC_Mobius}).

The third one is of special interest in the study of the CFT, since it can be used as the starting point to an algebraic description of LCFT in higher dimensions, despite this question being far from being understood.

Concerning the random measure which corresponds to the exponentiation of the Liouville field and formally defined by $e^{dbX}\mathrm{d}\lambda_g$, it has the following properties: 
\begin{prop}\label{prop:UAQS}
The Liouville measure $Z$ satisfies the following properties:\begin{itemize}
    \item \textbf{The total volume of the space}, $Z(\mathbb{R}^d)$, follows the Gamma distribution $\Gamma(\frac{\sum_i\alpha_i-Q}{b},\mu)$ in the sense that for any $F$ continuous bounded on $\mathbb{R}^+$,
    \begin{equation}
        \mathbb{E}_{b,\mu}^{(\bm{x},\bm{\alpha})}\left[F(Z(\mathbb{R}^d))\right]=\frac{\mu^{\frac{\sum_i\alpha_i-Q}{b}}}{\Gamma(\frac{\sum_i\alpha_i-Q}{b})}\int_{0}^{\infty}F(y)y^{\frac{\sum_i\alpha_i-Q}{b}-1}e^{-\mu y}dy.
    \end{equation}
    \item The law of \textbf{$\bm{Z}$ conditioned on the total mass being equal to $A$} is characterised by
    \begin{equation}\label{eq:UAQS}
        \mathbb{E}_{b,\mu}^{(\bm{x},\bm{\alpha})}\left[F(Z)\vert Z(\mathbb{R}^d)=A\right]=\frac{\expect{F(A\frac{Z_{\hat g}}{Z_{\hat g}(\mathbb{R}^d)})Z_{\hat g}(\mathbb{R}^d)^{-\frac{\sum_i\alpha_i-Q}{b}}}}{\expect{Z_{\hat g}(\mathbb{R}^d)^{-\frac{\sum_i\alpha_i-Q}{b}}}}
    \end{equation}
    where $dZ_{\hat g}=e^{dbH_{\hat{g}}}dM_{b,\hat g}$ and $F$ is bounded continuous (in the sense of weak convergence of measures) on the space of finite measures.
    \item The law of \textbf{$\bm{\frac ZA}$ conditioned on the total mass being equal to $A$} is independent of $A$ and is characterised by
    \begin{equation}
        \mathbb{E}_{b,\mu}^{(\bm{x},\bm{\alpha})}\left[F(\frac ZA)\vert Z(\mathbb{R}^d)=A\right]=\frac{\expect{F(\frac{Z_{\hat g}}{Z_{\hat g}(\mathbb{R}^d)})Z_{\hat g}(\mathbb{R}^d)^{-\frac{\sum_i\alpha_i-Q}{b}}}}{\expect{Z_{\hat g}(\mathbb{R}^d)^{-\frac{\sum_i\alpha_i-Q}{b}}}}.
    \end{equation}
\end{itemize} 
\end{prop}

\subsection{Liouville Conformal Field Theory for manifolds conformally equivalent to the sphere}
So far, we have introduced a rigorous definition of LCFT when the manifold being investigated was the sphere $(\mathbb{S}^d,g_0)$ or, by stereographic projection, the compactified Euclidean space with background metric $\hat g$. Actually the approach followed can be extended to the case of a manifold $(\mathcal{M},g)$ conformally equivalent to the sphere, which is the case for any compact Riemannian manifold without boundary, provided that it is simply connected and locally conformally flat.

To do so, consider a conformal diffeomorphism $\psi:(\mathcal{M},g)\rightarrow(\mathbb{S}^d,g_0)$. We define the law of the Liouville field $\phi_{\mathcal{M}}$ on $\mathcal{M}$ with marked points $(\bm x,\bm{\alpha})$ as
\begin{equation}
    \phi_{\mathcal{M}}:=\phi\circ\psi+Q\ln\norm{\psi'}
\end{equation}
where $\phi$ has the law of the Liouville field on $\mathbb{S}^d$ with marked points $(\psi (\bm x),\bm{\alpha})$ with $\psi(\bm x)\coloneqq(\psi(x_1),\cdots,\psi(x_N))$. By doing so, the property of conformal covariance of the GMC measure allows to say that the law of the corresponding Liouville measure on $\mathcal{M}$ is the same as the pushforward by $\psi$ of the Liouville measure on the sphere. This definition is consistent with the transformation rule~\eqref{conformal_covariance} on the sphere provided by the conformal covariance of the field.

All the previous properties (A-type anomaly, conformal covariance of the Vertex Operators) can be transposed to this new framework.

\section{Perspectives}
\subsection{The unit volume quantum sphere}
In Proposition~\ref{prop:UAQS} was introduced a fundamental probabilistic object corresponding to the higher-dimensional analogue of the so-called \textit{unit area quantum sphere}, which can be defined in the two-dimensional setting through (at least) three distinct approaches: 
\begin{itemize}
\item In his fundamental work \cite{S16}, Sheffield defined according to a limiting procedure what he called the \textit{unit area quantum sphere} and conjectured that this object should be somehow related to the limit of uniform quadrangulations. Later on, Duplantier, Miller and Sheffield in \cite{DMS14} provided a more explicit construction of these objects in terms of Bessel processes, and study their relationship with three key objects in the theory of random geometry: the Gaussian Free Field, the \textit{Schramm-Loewner Evolutions} and \textit{Continuum Random Trees}. 
\item A second approach was to consider the Liouville measures defined above by considering $Z$ conditioned on the total mass being equal to $1$. This is the approach we will develop here.
\item Another approach (developed for instance by Le Gall \cite{LeG13} and Miermont \cite{Mie13}) was to view this object as the scaling limit of random planar maps with the topology of the sphere. In these articles, the authors defined the \textit{Brownian map} as a metric space as opposed to the conformal structure the two perspectives we have presented so far rely on.
\end{itemize}

In the present framework, the proper definition of the measure $\mu^{UAQS}$ may therefore be given by setting $A=1$ in Equation~\eqref{eq:UAQS}. Namely:
\begin{equation}
    \mathbb{E}_{b,\mu}^{(\bm{x},\bm{\alpha})}\left[F(\mu^{UAQS})\right]\coloneqq\frac{\expect{F(\frac{Z}{Z(\mathbb{R}^d)})Z(\mathbb{R}^d)^{-\frac{\sum_i\alpha_i-Q}{b}}}}{\expect{Z(\mathbb{R}^d)^{-\frac{\sum_i\alpha_i-Q}{b}}}}
\end{equation}
where the $(\alpha_i)_{1\leq i\leq N}$ satisfy the bound of Theorem~\ref{partition_limit}. However, to ensure existence of the latter we may only assume that the quantity $\expect{Z(\mathbb{R}^d)^{-\frac{\sum_i\alpha_i-Q}{b}}}$ is finite, in which case one may extend the definition of the unit volume Liouville measure provided that the conditions \begin{equation}
    \forall {1\leq i\leq N}, \alpha_i<\frac Q2\quad\text{and}\quad Q-\sum_{i}\alpha_i<\frac{1}{b}\wedge\min\limits_{1\leq i\leq N}(Q-2\alpha_i)
\end{equation}
ensuring the finiteness of the quantity $\expect{Z(\mathbb{R}^d)^{-\frac{\sum_i\alpha_i-Q}{b}}}$ (see for instance \cite[Lemma 3.10]{DKRV16}) are satisfied.

Using the standard conventions of the mathematics literature (see Table~\ref{table_phy/maths} below), one may instead define the unit volume Liouville measure by setting 
\begin{equation}
    \mathbb{E}_{\gamma,\mu}^{(\bm{x},\bm{\alpha})}\left[F(\tilde{\mu}^{UAQS})\right]=\frac{\expect{F(\frac{Z_{\gamma}}{Z_{\gamma}(\mathbb{R}^d)})Z_{\gamma}(\mathbb{R}^d)^{-\frac{1}{\gamma}\left(\sum_i\alpha_i-2Q\right)}}}{\expect{Z_{\gamma}(\mathbb{R}^d)^{-\frac{1}{\gamma}\left(\sum_i\alpha_i-2Q\right)}}}
\end{equation}
where we have written $dZ_{\gamma}=e^{\gamma \Tilde{H}_{\hat{g}}}dM_{\gamma,\hat g}$ with $\Tilde{H}_{\hat{g}}=\sum_{i=1}^N\alpha_iG_{\hat{g}}(x,x_i)$ under the assumptions
\begin{equation}
    \forall {1\leq i\leq N}, \alpha_i< Q_{\gamma}\quad\text{and}\quad Q_{\gamma}-\sum_{i}\frac{\alpha_i}2<\frac{d}{\gamma}\wedge\min\limits_{1\leq i\leq N}(Q_{\gamma}-\alpha_i).
\end{equation}
The first assumption corresponds to integrability of the GMC near the singular points $x_1,\cdots,c_N$ like in Lemma~\ref{lemma_conv}; the second bound correspond to the finiteness of moments of the GMC measure (via a straightforward adaptation of~\cite[Lemma 3.10]{DKRV16}). 

A particularly interesting case is the one where we have fixed three marked points with weight $\gamma$. In the two-dimensional setting, this random measure is conjectured to be the limit of some models of random planar maps conditioned to have total area $1$ (see \cite[Subsection 5.3]{DKRV16} for a precise statement), and corresponds in some sense to the unit area quantum sphere defined in \cite{DMS14}: a precise notion of equivalence is proved in \cite{AHS16}. It would be interesting to provide a similar definition for quantum spheres involving scaling limits of discrete models in higher dimension and to relate it to the objects introduced in \cite{LM19}, which should describe higher-dimensional analogues of the Brownian map.

\subsection{The semi-classical limit}\label{subsec:semi_classical}
We have seen that in the classical theory, the critical points of the action functional (to which we have added conical singularities) introduced above correspond to metrics with constant negative $\mathcal{Q}$-curvature and conical singularities prescribed by the marked points $(\bm x,\bm{\alpha})$. The semi-classical limit consists of the study of the asymptotic properties of the (quantum) Liouville field introduced when the quantum parameter $b$ converges toward $0$: in this context one should observe a concentration phenomenon of the Liouville field around the classical solution of the problem introduced in Theorem~\ref{classical_conical}. 

More precisely, let us denote by $\phi$ the Liouville field and $\phi^*$ the solution of the problem in Theorem~\ref{classical_conical}. As we have seen before, in the classical theory one should think of $\phi^*$ as $b\phi$: therefore in the semi-classical limit one is interested in the asymptotic properties of the rescaled field $b\phi$. In this regime, one should consider the rescaled weights $\alpha_k:=\frac{\chi_k}{b}$ (the corresponding Vertex Operators are usually referred to as \textit{heavy operators} in the physics literature) as well as a specific value for the cosmological constant: $\mu=\Lambda b^2$.
By doing so we see that the action corresponds to the variational formulation of the problem of Theorem~\ref{classical_conical} and therefore in the semi-classical limit one should recover the classical solution to this problem thanks to a saddle point method similar to the one developed in \cite{LRV19}. 

Indeed in this regime the law of the Liouville field can be rewritten under the form 
\[
\expect{F(b\phi)}=\int_{\mathbb{R}}\frac{\expect{F\left(bX_{\hat g}-\frac1d\ln Z_{\hat g}+\frac 12\ln \hat{g}w +\frac1d c+b^2(\frac12\ln \hat g-\tilde{C}_{\hat g})\right)Z_{\hat g}^{-s}}}{\expect{Z_{\hat g}^{-s}}}\frac{(\Lambda b^{-2})^se^{cs}e^{-\Lambda b^{-2} e^c}}{\Gamma(s)}dc
\]
where we have set $s=\frac{\sum_k \chi_k-1-b^2}{b^2
}$ and $w(x)=e^{\sum_k 4\chi_k G_{\hat g}(x,x_k)}$. 

When letting $b$ go to $0$ we see that, on the one hand, the integral expression involving the variable $c$ will converge towards the quantity \[\frac{\expect{F\left(bX_{\hat g}-\frac1d\ln Z_{\hat g}+\frac 12\ln \hat{g}w +\frac 1d\ln \frac{\sum_k\chi_k-1}{\Lambda})\right)Z_{\hat g}^{-s}}}{\expect{Z_{\hat g}^{-s}}}.
\]

On the other hand, it seems natural to expect that the random field $bX_{\hat{g}}-\frac 1d\ln Z_{\hat g}$ under the probability measure weighted by $Z_{\hat g}^{-s}$ converges in probability to the deterministic field $\overline{h}$ defined by
\[\exp\left(d\overline{h}\right):=\frac{\exp(dh)}{\int_{\mathbb{R}^d}\exp(d(h+\frac12\ln\hat g w))\mathrm{d}\lambda}
\]

where $h=h_0\circ\psi^{-1}$, with $\psi$ the standard stereographic projection and $h_0$ the unique $H^{\frac{d}{2}}(\mathbb{S}^d,g_0)$ solution of
\[
\begin{cases}
\mathcal{P}_{0}h_0=2\gamma_d(1-\sum_k\chi_k)\left(\frac{\exp\left(d(h_0+\frac12\ln w\circ\psi)\right)}{\int_{\mathbb{S}^d}\exp\left(d(h_0+\frac12\ln w\circ\psi)\right)\mathrm{d}\lambda}-\frac{1}{\norm{\mathbb S^d}}\right)\\
\int_{\mathbb{S}^d}h_0\mathrm{d}\lambda=0.\\
\end{cases}
\]
The map $\psi$ being conformal we see that $\overline{h}$ is a solution of
\[
(-\Delta)^{\frac d2}\overline{h}=2\gamma_d(1-\sum_k\chi_k)\left(e^{d(\overline{h}+\frac12\ln \hat gw)}-\frac{\hat g^{\frac d2}}{\norm{\mathbb S^d}}\right).
\]

To summarize, the field $b\phi$ in the semi-classical limit should converge to the deterministic quantity \[\phi:=\overline{h}+\frac12\ln\hat g w+\frac1d\ln \frac{\sum_k\chi_k-1}{\Lambda}.
\]

It is then easily checked that $\phi$ is indeed a solution to the constant negative curvature problem
\[
\begin{cases}
(-\Delta)^{\frac{d}{2}}\phi+2\gamma_d\Lambda e^{d\phi}&=2\gamma_d\sum_k \chi_k \delta(x-x_k)\\
\phi\sim -2\ln\norm{x}&\text{as }x\rightarrow \infty.\\
\end{cases}
\]

As in the two-dimensional setting, we also expect that a second order expansion of the field $b\phi$ could be done (with a result similar to \cite[Theorem 2.4]{LRV19}) as well as a study of the semi-classical limit of the correlation function. We stress that there should not be so much additional difficulties compared to the case treated in~\cite{LRV19} but would involve some technicalities to be taken care of. It would also be interesting to derive a statement analogous to the Takhtajan-Zograf theorem~\cite[Theorem 1]{TZ02} in this higher-dimensional context too.

\subsection{The theory with a boundary}
In full analogy with the two-dimensional case, a challenging problem would be to extend the path integral approach to manifolds with boundary. In doing so, one could expect a similar construction involving GMC measures (a bulk measure and boundary measures) whose interactions can be parametrised with cosmological constants. In two dimensions the rigorous construction is done in \cite{HRV16}. The construction would also allow to define LCFT in odd-dimensional manifolds without boundaries, by viewing them as boundaries of even-dimensional manifolds. 

However providing an explicit formulation in terms of an action functional in the higher-dimensional case is still an open problem, closely related to the problem of constructing boundary operators associated to GJMS operators. Partial results are known in the four and six-dimensional cases (see \cite{CL18} for instance), but in the general this remains an open question, even from the perspective of the physics literature.

\appendix
\section{Auxiliary computations and proofs}
\subsection{Correspondence between different conventions for Liouville Conformal Field Theory}\label{table_phy/maths}
\hspace{0.05cm}\\
\begin{table}[h]
\begin{tabular}{|l|c|c|c|}
  \hline
   & Physics & Mathematics & Relationship \\
  \hline
  Coupling constant & $b\in(0,1)$ & $\gamma\in(0,\sqrt{2d})$ & $\gamma=b\sqrt{2d}$ \\
  \hline
   Background charge & $Q_b=b+\frac1b$ & $Q_{\gamma}=\frac{\gamma}{2}+\frac{d}{\gamma}$ & $Q_{\gamma}=\sqrt{\frac2d}Q_b$\\
  \hline
  Cov. of Liouville field & $\sim \frac{2}{d} \ln\frac{1}{\norm{x-y}}$ & $\sim \ln\frac{1}{\norm{x-y}}$ & $\varphi_{\gamma}=\sqrt{\frac{d}{2}}\varphi_b$\\
  \hline
  Vertex Operators & $e^{d\alpha\varphi_b(x)}$ & $e^{\beta\varphi_{\gamma}(x)}$ & \\
  \hline
  Conformal dimension & $d\alpha(Q_b-\alpha)$ & $\frac{\beta}{2}(Q_{\gamma}-\frac{\beta}{2})$ & \\
  \hline
  Seiberg bounds & $\sum_i\alpha_i>Q_b, \alpha_i<\frac{Q_b}{2}$ & $\sum_i\beta_i>2Q_{\gamma}, \beta_i<Q_{\gamma}$ &\\ 
  \hline
  
\end{tabular}

\caption{Different conventions for LCFT in even dimension $d$.}
\end{table}

\subsection{Auxiliary properties}
The first property is a special case of the celebrated Liouville's theorem which asserts that any harmonic function bounded and defined on the whole Euclidean space must be constant. In our simplified context the result is rather elementary:
\begin{lemma}\label{radial_harmonic}
Assume that $F$ is a radial, smooth function on $\mathbb{R}^d$ such that $\Delta^{\frac d2}F=0$.

If $F$ is bounded then $F$ must be constant.
\end{lemma}
\begin{proof}
Since $F$ is radial one can express its Laplacian as \[
\Delta F = \frac{1}{r^{d-1}}\frac{d}{dr}r^{d-1}F'(r).
\]
As a consequence if $F$ is such that $\Delta^{h}F=0$ we see that $F$ is polynomial in the variables $(r,\frac{1}{r},\ln r)$. The assumption that $F$ is bounded close to the origin implies that it is polynomial in $(r,\ln r)$; its behaviour close to $+\infty$ implies that it is constant.
\end{proof}
The second property is a very-well known property of Gaussian vectors and processes, usually referred to as Girsanov or Cameron-Martin theorem~\cite[Chapter VIII]{RY91}. Even if the theorem as stated below does not apply to our context (because of the smoothness assumption), it can be easily adapted to our purpose thanks to the regularization procedure for our log-correlated field:
\begin{theorem}[Girsanov theorem]\label{Girsanov}
Let $D$ be a subset of $\mathbb{R}^d$ and $(X(x))_{x\in D}$ be a smooth centered Gaussian field. If $Z$ is a Gaussian variable belonging to the $L^2$ closure of the subspace spanned by $(X(x))_{x\in D}$ then, for any bounded functional $F$ over the space of continuous functions,
\begin{equation}
\expect{e^{Z-\frac{\expect{Z^2}}{2}}F(X(x))_{x\in D}}=\expect{F\left(X(x)+\expect{ZX(x)}\right)_{x\in D})}.
\end{equation}
\end{theorem}

\subsection{Classical Liouville Conformal Field Theory: some proofs}
As we will see, the construction of a solution for Theorem~\ref{classical_conical} follows closely the lines of Subsection~\ref{subsec:semi_classical}. We adapt this reasoning to the setting of the sphere by introducing throughout the subsequent proofs the function on $\mathbb{S}^d$ given by $w_0(x):=e^{\sum_{k}4\chi_k \ln{\frac{1}{\norm{x-x_k}}}}$ (analogous to the map $w$ of Subsection~\ref{subsec:semi_classical} which was defined on $\mathbb{R}^d$). Similarly we will construct solutions $\overline{h}_0$ and $\phi_0$ which are counterparts of $\overline{h}$ and $\phi$ constructed highlighted in Subsection~\ref{subsec:semi_classical}. 

\begin{proof}[Proof of the Moser-Trudinger-type inequality (Proposition~\ref{Mos_Tru})]\label{classical_proofs}
Consider $f$ a smooth function with vanishing mean over $\mathbb{S}^d$ and $q>1$ such that $q\chi_k<\frac 12$ for all $k$. Then for $p$ such that $1=\frac 1p+\frac 1q$ H\"older inequality shows that
\[
\int_{\mathbb{S}^d}e^{d(f+\frac12\ln w)}\mathrm{d}\lambda_{g_0}\leq \left(\int_{\mathbb{S}^d}e^{dpf}\mathrm{d}\lambda_{g_0}\right)^{\frac 1p}\left(\int_{\mathbb{S}^d}e^{\frac {dq}2\ln w}\mathrm{d}\lambda_{g_0}\right)^{\frac 1q}.
\]
Note that the second integral is finite since we have imposed the condition $q\chi_k<\frac{1}{2}$ so the singularities of $e^{\frac {dq}2\ln w}$ are integrable.
Therefore the standard Moser-Trudinger inequality \cite[Theorem 1]{B93} applied to $pf$ yields the inequality
\[
\ln \int_{\mathbb{S}^d}e^{d(f+\frac12\ln w)}\mathrm{d}\lambda\leq c+C\int_{\mathbb{S}^d} f\mathcal{P}_0 f\mathrm{d}\lambda.
\]
where the constants $c$ and $C$ are positive. If $f$ has non-zero mean the inequality extends in a straightforward way:
\[
\ln \int_{\mathbb{S}^d}e^{d(f+\frac12\ln w)}\mathrm{d}\lambda_{g_0}\leq c+C\int_{\mathbb{S}^d} f\mathcal{P}_0 f\mathrm{d}\lambda_{g_0} + d\int_{\mathbb{S}^d} f\mathrm{d}\lambda_{g_0}
\]
Since this inequality holds true for any smooth function over $\mathbb{S}^d$ and the right-hand-side quantity is bounded by a multiple of the Sobolev norm, it extends to the Sobolev space $H^{\frac d2}(\mathbb{S}^d,g_0)$.
\end{proof}

\begin{proof}[Existence and uniqueness of a constant negative $\mathcal{Q}$-curvature metric (Theorem~\ref{classical_conical})]
\hspace*{0.05cm}\\
We start by noting that the map $w_0$ is such that 
\begin{equation}
    \mathcal{P}_0\frac12\ln w_0 =2\gamma_d \sum_k\chi_k(\delta(x-x_k)-\frac{1}{\norm{\mathbb{S}}^d}).
\end{equation}
To see this, simply note that for $f:\mathbb{S}^d\rightarrow\mathbb{R}$ such that $f\circ\psi^{-1}$ is smooth and compactly supported in $\mathbb{R}^d$ (with $\psi:\mathbb{S}^d\rightarrow\mathbb{R}^d$ the stereographic map)  one has that 
\begin{align*}
    \int_{\mathbb{S}^d}\ln\frac{1}{\norm{x-y}}\mathcal{P}_0f(y)\mathrm{d}\lambda_{g_0}(y)&= \int_{\mathbb{R}^d}G_{\hat{g}}(\psi(x),y)\mathcal{P}_{\hat{g}}(f\circ\psi^{-1})(y)\mathrm{d}\lambda_{\hat{g}}(y)\\
    &=\gamma_d\left(f(x)-m_{\hat{g}}(f\circ\psi^{-1})\right)\\
    &=\gamma_d\int_{\mathbb{S}^d}(\delta(x-y)-\frac{1}{\norm{\mathbb{S}^d}})f(y)\mathrm{d}\lambda_{g_0}(y).
\end{align*}
Here we have used that for $x,y$ in $\mathbb{S}^d$, $\ln\frac{1}{\norm{x-y}}=\ln{\frac{1}{\norm{\psi(x)-\psi(y)}}}-\frac{1}{4}\Big(\ln\hat g(\psi(x))+\ln\hat g(\psi (y))\Big)$ and that since the stereographic map is conformal and $\hat g$ is the pushforward of $g_0$ by $\psi$, $\mathcal{P}_{0}(f)\circ\psi^{-1}=\mathcal{P}_{\hat g}(f\circ\psi^{-1})$.
This equality extends to any smooth function $f$ over $\mathbb{S}^d$ by using a truncation. This means that in the sense of distributions \[
\mathcal{P}_0\ln\frac{1}{\norm{x-y}}=\gamma_d(\delta(x-y)-\frac{1}{\norm{\mathbb{S}}^d}).
\]

\begin{itemize}
    \item To prove existence, we will use a variational formulation of the problem.
\end{itemize}
Let us introduce on $H^{\frac d2}(\mathbb{S}^d,g_0)$ the functional given by
\begin{equation}
J(h):= \frac{d}{4\gamma_d}\int_{\mathbb{S}^d}h\mathcal{P}_0h+2(d-1)!c_0 h\mathrm{d}\lambda_{g_0}-c_0\ln \int_{\mathbb{S}^d}e^{d(h+\frac12\ln w_0)}\mathrm{d}\lambda_{g_0}
\end{equation}
where recall that $\gamma_d=\frac{(d-1)!\norm{\mathbb{S}^d}}{2}$; we have also introduced $c_0:=\sum_k\chi_k-1>0$. Note that this functional is indeed well-defined thanks to the inequality~\eqref{Mos_Tru}. 

Then critical points of the functional $J$ correspond to variational solutions of
\begin{equation}\label{Lio_equ}
    \mathcal{P}_{0}h=-2\gamma_dc_0\left(\frac{e^{d(h+\frac12\ln w_0)}}{\int_{\mathbb{S}^d}e^{d(h+\frac12\ln w_0)}\mathrm{d}\lambda}-\frac{1}{\norm{\mathbb S^d}}\right).
\end{equation}
Moreover the functional $J$ is unchanged if $h$ is shifted by an additive constant: we can therefore consider a minimizing sequence $(h_k)_{k\in\mathbb{N}}$ of elements of $H^{\frac d2}(\mathbb{S}^d,g_0)$ with zero mean on the sphere. In that case we can apply the Moser-Trudinger-type inequality (Proposition \ref{Mos_Tru}) to get that for any $k$ we have 
\[
    \int_{\mathbb{S}^d}h_k\mathcal{P}_0h_k\mathrm{d}\lambda_{g_0}\leq c J(h_k)+C
\]
for some positive constants. 
Then a recursive application of the Poincar\'e inequality on the sphere shows that 
\[
    \int_{\mathbb{S}^d}\norm{h_k}^2\mathrm{d}\lambda_{g_0}\leq A+B\int_{\mathbb{S}^d}h_k\mathcal{P}_0h_k\mathrm{d}\lambda_{g_0}
\]
for positive constants. As a consequence the sequence of the $(h_k)_{k\in\mathbb{N}}$ is bounded in $H^{\frac d2}(\mathbb{S}^d,g_0)$ so it admits a subsequence that converges weakly in $H^{\frac d2}(\mathbb{S}^d,g_0)$ towards some $h_0$ which is thus a critical point of $J$, and therefore a variational solution of Equation~\eqref{Lio_equ}. However we can deduce from~\cite[Theorem 1.1]{UV00} that $h_0$ is actually smooth on $\mathbb{S}^d$, and therefore that $h_0$ does indeed solve Equation~\eqref{Lio_equ} in the strong sense.

To finish up, we define $\overline{h}_0$ by setting $e^{d\overline{h}_0}:=\frac{e^{d(h+\frac12\ln w_0)}}{\int_{\mathbb{S}^d}e^{d(h+\frac12\ln w_0)}\mathrm{d}\lambda_{g_0}}$ and $\phi_0:=\overline{h}_0+\frac{1}{d}\ln \frac{c_0}{\Lambda}$. Then we have that
\[
\mathcal{P}_0 \phi_0+2\gamma_d\Lambda e^{du_0}+(d-1)!= 2\gamma_d\sum_k\chi_k\delta(x-x_k).
\]
Put differently, the conformal metric $e^{2\phi_0}g_0$ has constant negative $\mathcal{Q}$-curvature $-2\gamma_d\Lambda$ and conical singularities given by $(\bm x,\bm{\chi})$.
\begin{itemize}
    \item For the uniqueness part, we start by considering $h_1$ and $h_2$ two solutions in $H^{\frac d2}(\mathbb{S}^d,g_0)$ of the variational problem \eqref{Lio_equ} (which are actually smooth by ~\cite[Theorem 1.1]{UV00}). Then one has that
\end{itemize}
\begin{align*}
    &\int_{\mathbb{S}^d}(h_1-h_2)(x)\mathcal{P}_0(h_1-h_2)(x)\mathrm{d}\lambda_{g_0}(x)\\
    &=2\gamma_d(1-\sum_k\chi_k)\int_{\mathbb{S}^d}(h_1-h_2)(x)\left(\frac{e^{d(h_1+\frac 12\ln w_0)(x)}}{\int_{\mathbb{S}^d}e^{d(h_1+\frac 12\ln w_0)}\mathrm{d}\lambda_{g_0}}-\frac{e^{d(h_2+\frac 12\ln w_0)(x)}}{\int_{\mathbb{S}^d}e^{d(h_2+\frac 12\ln w_0)}\mathrm{d}\lambda_{g_0}}\right)\mathrm{d}\lambda_{g_0}(x)\\
    &=2\gamma_d(1-\sum_k\chi_k) \int_{\mathbb{S}^d}(h_1-h_2)(x)\left(\int_0^1\frac d{dt}\frac{e^{d(h_2+t(h_1-h_2)+\frac 12\ln w_0)(x)}}{\int_{\mathbb{S}^d}e^{d(h_2+t(h_1-h_2)+\frac 12\ln w_0)}\mathrm{d}\lambda_{g_0}}dt\right)\mathrm{d}\lambda_{g_0}(x)\\
   &=2\gamma_d(1-\sum_k\chi_k)\int_0^1 \left(\int_{\mathbb{S}^d}(h_1-h_2)^2(x)d\mu_t(x)-\left(\int_{\mathbb{S}^d}(h_1-h_2)(x)d\mu_t(x)\right)^2\right)dt
\end{align*}
where $d\mu_t(x):=\frac{e^{d(h_2+t(h_1-h_2)+\frac 12\ln w_0)(x)}\mathrm{d}\lambda_{g_0}(x)}{\int_{\mathbb{S}^d}e^{d(h_2+t(h_1-h_2)+\frac 12\ln w_0)}\mathrm{d}\lambda_{g_0}}$ is a probability measure on $\mathbb{S}^d$ (note that the singularities coming from $w_0$ are integrable since we assume the Seiberg bounds to hold). The last equality is obtained by applying Fubini-Tonelli theorem for the first term (the integrand is positive) and Fubini-Lebesgue for the second one ($h_1-h_2$ is smooth thus the integral is absolutely convergent). As a consequence (Cauchy-Schwarz inequality) the integrand in the $t$ variable is non-negative; since we have assumed the second Seiberg bounds to hold (\textit{i.e.} $\sum_k\chi_k>1)$ this shows that  \[\int_{\mathbb{S}^d}(h_1-h_2)(x)\mathcal{P}_0(h_1-h_2)(x)\mathrm{d}\lambda_{g_0}(x)\leq0.
\] 
Since the operator $\mathcal{P}_0$ is non-negative this implies that $h_1-h_2$ is in the kernel of $\mathcal{P}_0$, that is $h_1$ and $h_2$ differ by a constant.

Now if we consider $u_0$ to be a solution of the constant $\mathcal{Q}$-curvature problem and set $h_0:=u_0-\frac12\ln w_0$ we see by using that $\int_{\mathbb{S}^d}e^{du_0}\mathrm{d}\lambda=\frac{c_0}{\Lambda}$ that $h_0$ is a solution of the variational problem for which we have just proved uniqueness up to a constant. Since this constant is fixed in the constant-curvature problem by the value of the total integral we see that $u_0$ is uniquely determined.
\end{proof}

\addcontentsline{toc}{section}{References}
\bibliographystyle{plain}
\bibliography{biblio}
\end{document}